\documentclass{amsart}

\usepackage{amsfonts}
\usepackage{amsmath}
\usepackage{amssymb, esint}
\usepackage{amsthm,color}
\usepackage{enumerate}
\usepackage{mathtools}
\usepackage{eurosym}
\usepackage[hidelinks, bookmarksdepth=3]{hyperref}

\usepackage{url}
\usepackage[hidelinks, bookmarksdepth=3]{hyperref}
\usepackage{xcolor}
\usepackage{cancel}
\usepackage{ulem} 
\usepackage{caption,subcaption}
\usepackage{tikz}
\usepackage{tikz-3dplot}

\usepackage{imakeidx}
\usepackage{tabularx, booktabs}
\usepackage{multicol}

\usepackage{graphicx}
\usepackage{amssymb,amscd,amsthm,amsxtra,esint}
\usepackage{latexsym}
\usepackage{epsfig,esint}
\newtheorem{thm}{Theorem}[section]
\newtheorem{cor}[thm]{Corollary}
\newtheorem{lem}[thm]{Lemma}
\newtheorem{prop}[thm]{Proposition}
\theoremstyle{definition}
\newtheorem{defn}[thm]{Definition}
\theoremstyle{remark}
\newtheorem{rem}[thm]{Remark}
\numberwithin{equation}{section}

\newcommand{\be}{\begin{equation}}
\newcommand{\ee}{\end{equation}}

\newcommand{\R}{\mathbb R}

\newcommand{\eps}{\varepsilon}

\newcommand{\p}{\partial}

\newcommand{\comment}[1]{}
\usepackage{geometry}
\definecolor{strongmagenta}{RGB}{180,0,180}

\begin{document}

\title[Triple Junctions]{Triple Junctions in Multi-Phase Obstacle type  Problems}
\author{Daniela De Silva}
\address{Department of Mathematics, Barnard College, Columbia University, New York, NY 10027}
\email{\tt  desilva@math.columbia.edu}
\author{Ovidiu  Savin}
\address{Department of Mathematics, Columbia University, New York, NY 10027}\email{\tt  savin@math.columbia.edu}
\author[Shahgholian]{Henrik Shahgholian%\,\orcidlink{0000-0002-1316-7913}
}
\address{Department of Mathematics, KTH Royal Institute of Technology, SE-10044 Stockholm, Sweden 
%\& Faculty of Mathematics and Mechanics,  Yerevan State University, Armenia.
}
\email{henriksh@kth.se}

\begin{abstract} We study the regularity of the free boundary for a multi-phase obstacle-type problem associated to quadrature domains, which arises from the minimization of the energy 
$$J(U,F):= \int_{B_1} (\frac 1 2 |\nabla U|^2 + F \cdot U)dx$$
among all competitors $U=(u^1,\ldots, u^N)$ with non-negative components $u^i \geq 0$ with disjoint supports, and constant forcing terms $f^i>0.$ We focus on the geometry of the free boundary at points where at least three distinct components meet (triple junctions).  \end{abstract}

%\thanks{}
\subjclass{35R35}%
\keywords{Multi-phase free boundary problem, Triple-junction points, Regularity theory}

\maketitle

\maketitle

\section{Introduction}

In this paper we study the regularity of the free boundary arising in a multi-phase obstacle-type problem associated to quadrature domains. Precisely, we consider minimizers of the energy functional  
$$J(U,F):= \int_{B_1} (\frac 1 2 |\nabla U|^2 + F \cdot U)dx$$
among all competitors $U=(u^1,\ldots, u^N)$ which have non-negative components $u^i \geq 0$   whose positivity sets are pairwise disjoint. In other words, $U$ maps  $B_1 \subset \R^n$ to the union of the non-negative coordinate axes in $\R^N$.

Specifically, we minimize $J$ over the class $\mathcal A$ defined as
$$\mathcal A:=\left \{U \in H^1(B_1)| \quad u^i \cdot u^j = 0 \quad \mbox{for} \  i \neq j, \quad u^i \geq 0 \right \},$$ 
with a given forcing term $F=(f^1, \ldots, f^N) \in L^2(B_1)$, $B_1 \subset \R^n$. There is a free boundary associated to this minimization problem, consisting of the boundaries of the supports of each  $u^i$, that is 
$$\Gamma(U):= \p\{|U|>0\} \cap B_1.$$ 

The case $F\equiv 0$ has been widely investigated and appears in a variety of contexts: as harmonic maps into singular spaces \cite{GS,S}, as limit of competition-diffusion models \cite{CTV1,CTV2,CTV4}, optimal partitions for spectral  problems \cite{BBH,CL2,CTV3}, and as double-valued harmonic functions \cite{A,CL1, DLS, SW}. In particular, minimizers are known to be Lipschitz continuous and the free boundary can be decomposed into a regular part consisting of analytic hypersurfaces separating adjacent supports and a singular part in which $U$ has a homogeneous behavior of homogeneity $\gamma >1$. For a more detailed explanation, we refer to the recent survey by Ognibene and Velichkov \cite{OV1} and the references therein.

The minimization problem in the class $\mathcal A$, for an energy involving both the Dirichlet part and a more general potential term $F(x,U)$, was considered by Conti, Terracini, and Verzini in \cite{CTV4} in the modeling of reaction-diffusion phenomena. The energy $J(\cdot, F)$ above with $f^i > 0$ appears in the work of the third author with Arakelyan \cite{AS}, in the construction of multi-phase quadrature domains.

\subsection{Main Results.}\label{subsec:main}

 The main issue under investigation in this paper is the classification of global solutions and local regularity properties of the free boundary of minimizers of the energy $J$. We initially address the basic questions of optimal regularity, compactness, uniqueness, and non-degeneracy of minimizers. Some of these basic results can be found in the literature, possibly in a slightly different setting. Thus, for completeness, we provide the details.
 
 Our main purpose is to study the free boundary $\Gamma(U)$ in the case when $F > 0$ is constant. This is done by a blow-up procedure which depends on the growth of $U$ near the free boundary. Precisely, let $0 \in \Gamma(U)$, and denote by
$$a_r:= \max_{B_r} |U|.$$
We show that if 

\begin{equation} \label{c1intro}
    \lim_{r\to0} \frac{a_r}{r^2} = \infty,
\end{equation}
then,  what we call  {\it type I  rescalings}, which  normalize the $L^\infty$ norm must converge (up to a subsequence) 
$$U_r(x):= a_r^{-1} U(rx) \to \mathcal U_0,$$
and the limit $\mathcal U_0$ is a nontrivial homogeneous minimizer of $J_0=J(\cdot,0)$. We refer to $\mathcal U_0$ as a {\it type I (tangent) cone}.
On the other hand if \eqref{c1intro} does not hold, then 
\begin{equation} \label{c2intro}
\liminf_{r\to0} \frac{a_r}{r^2} < \infty,
\end{equation}
and the quadratic rescalings, what we call \emph{type II  rescalings},    converge (up to a subsequence) $$U_r(x):= r^{-2} U(rx) \to U^F_0,$$
where the limit $U^F_0$ is a nontrivial minimizer homogeneous of degree two of $J(\cdot,F)$.
We refer to $U^F_0$ as a {\it type II (tangent) cone}.

We focus our attention on free boundary points where at least 3 components meet. Otherwise, when there are at most two nonzero components of $U$ in a neighborhood of a free boundary point, say $u^1$ and $u^2$, the problem can be reduced to a scalar one by considering the function
$$u:= u^1-u^2,$$
which solves the two-phase obstacle problem
$$\triangle u= f_1 \chi_{\{u>0\}} - f_2 \chi_{\{u<0\}}.$$
For the regularity of the free boundary for this scalar problem we refer to the book by Petrosyan, Shahgholian, and Uraltseva \cite{PSU}.

We establish that when $n=2$, there are at most 3 components that can meet at a free boundary point and we refer to such a point that belongs to the boundaries of 3 different components as a {\it triple junction}. Moreover, there are only two possible blow-up cones at a triple junction, either the $3/2$ homogeneous harmonic which is of type I, or an explicit $2$-homogeneous solution $U_0^F$ which is of type II. The phases of $U_0^F$ occupy three sectors of angles $\theta_i \in (\frac \pi 2, \pi)$, $i=1,2,3,$ which are determined uniquely from the components of $F$ that correspond to the three sectors. Precisely, we have the following.

\begin{thm}\label{Pintro} Assume $n=2$, $F>0$ constant, $N \ge 3$, and that $0$ belongs to the free boundaries of at least three components of a minimizer $U \in \mathcal A$ of $J(\cdot, F)$. Then there are exactly three nontrivial components in a neighborhood of $0$. Moreover, there are only two possible blow-ups (up to rotations and reflections) at $0$ 
$$|\mathcal U_0|=r^{3/2} |\cos \frac{3 \theta}{2}| \quad \mbox{(type I), or} \quad \quad U_0^F \quad \mbox{as in \eqref{uof} (type II)}.$$
\end{thm}

We also establish the uniqueness of the tangent cones at a triple junction which constitutes the main part of the paper. As a consequence we obtain the free boundary regularity in a neighborhood of a triple junction in 2 dimensions. 

\begin{thm}[Regularity of the free boundary at triple junctions]\label{localclassintro}
Assume $n=2$, $F>0$ constant, and that $x_0$ is a triple junction point for  a minimizer $U \in \mathcal A$ of $J(\cdot, F)$. Then in $B_{r_0}(x_0)$, for some $r_0$ small, $\Gamma(U)$ consists of three $C^{1,\alpha}$ curves starting at $x_0$ and meeting at either $120^\circ$ angles, or at $\theta_i$, $i=1,2,3$ angles obtained from $F$ by the compatibility equation
\begin{equation}\label{uofintro}
 f^1\tan\theta_1=f^2\tan\theta_2=f^3\tan\theta_3, \quad \sum \theta_i=2\pi, \quad \theta_i \in (\frac \pi 2, \pi).
 \end{equation}
\end{thm}

The two different behaviors of the free boundary at a triple junction correspond to the two types of tangent cone. In the symmetric case when the 3 nontrivial components of $F$ are all equal, the angles $\theta_i$ are equal to $120^\circ$ as well.

We show further that the cone $\mathcal U_0$ is stable in the sense that a small perturbation of its boundary data on $\p B_1$ produces a solution with a triple junction point that has a tangent cone of type I. On the other hand the type II cone $U_0^F$ is unstable, since a small perturbation of its boundary data may produce solutions with no triple junctions and fattening of the zero phase, or solutions that have a triple junction point with a type I tangent cone. In other words, an arbitrarily small perturbation of a triple junction with angles $\theta_i$ may produce a triple junction with $120^\circ$ angles. This shows that the $C^{1,\alpha}$ norm of the free boundaries in the second alternative cannot be bounded above, unless $F$ is symmetric. 

In order to understand perturbations of the type II cones we need to classify global solutions that are asymptotic to $U_0^F$ at infinity. It turns out that there is a one-parameter family of such solutions (modulo rigid motions) that we denote by $U_t^F$, for some parameter $t \in \R$. The solutions $U^F_t$ with $t>0$ (or with $t<0$) are obtained from each other after quadratic rescalings. The family $U_t^F$ can be constructed explicitly in the symmetric case when $F=(1,1,1)$. For the general case, one needs to construct global solutions with prescribed asymptotics at infinity by a limiting procedure,
see Section \ref{sec:non-sym-F}.

The family $U^F_t$ plays an important role in understanding the regularity properties of triple junctions for our problem. Using the analogy with minimal surface theory, it can be thought as providing the Hardt-Simon foliation near the cone $U_0^F$, with $U^F_t$ for $t>0$, and $t<0$, representing the leaves on each side of the cone. 

We provide a complete classification of global solutions for $n=2$, $N=3$ that have at most polynomial growth at infinity, and that have connected phases in contact with each other.

\begin{thm}[Classification of Global Solutions in 2d]\label{mainintro} For $n=2, N=3$, let $U \in \mathcal A$ be a global minimizer of $J(\cdot, F)$, $F> 0$ constant, with polynomial growth at infinity, $$|U(x)| = O(|x|^m),  \quad \text{for some $m \in \mathbb N$}.$$ Assume $$\{u^i>0\} \quad  \text{are connected,}$$ and $$\p\{u^i > 0\} \cap \p\{u^j >0\} \neq \emptyset \quad \text{for all $i,j \in \{1,2,3\}$.}$$ Then, $U$ must coincide with $U^F_t$ for some $t$ (modulo rigid motions).
\end{thm}

In view of Theorem \ref{mainintro}, the family $U^F_t$ models the local behavior of solutions defined in $B_1$ that have phases in contact with each other in a sufficiently small ball around the origin. 
This leads to a more accurate description than the one in Theorem \ref{localclassintro}, in the sense that if three phases are present in a small ball then the solution can be represented as a $C^{1,\alpha}$ diffeomorphism of one of the global solutions $U_t^F$ for some $t$ small (see Theorem \ref{thm:best}).

We also study the behavior of the free boundary in arbitrary dimensions $n \ge 3$, at points that have a two-dimensional tangent cone. In particular we say that $x_0 \in \Gamma(U)$ is a $Y$-singularity of type I or of type II if a blow-up cone coincides (modulo rotations) with the trivial extensions to $\R^n$ of the two-dimensional cones $\mathcal U_0$ and $U_0^F$ from Theorem \ref{Pintro}.

\begin{thm}\label{typeIgeneralintro} Let $x_0 \in \Gamma(U)$ be a type I Y-singularity. Then, the blow-up profile is unique and in a neighborhood of $x_0$ the free boundary $\Gamma(U)$ consists of three $C^{1,\alpha}$ hypersurfaces that meet at $120^\circ$ on a $C^{1,\alpha}$ codimension 2 curve made of type I Y-singularities (for any $\alpha<1/2$).
\end{thm}

\begin{thm}\label{typeIIgeneralintro} Let $x_0 \in \Gamma(U)$ be a type II Y-singularity. Then, the blow-up profile is unique. Moreover, all the triple junction points near $x_0$ belong to a $C^{1,\alpha}$ codimension-two submanifold  (for any $\alpha<1/2$). 
\end{thm}

\subsection{Related results.}\label{subsec:related} 
In the context of area-minimizing surfaces triple junction singularities were investigated in the pioneering   work of Jean Taylor \cite{T1, T2}, for hypersurfaces in $\R^3$, and by
Leon Simon \cite{Sim} in any dimension $n\geq 3$. In the context of other free boundary problems, triple junction configurations appear on the
graphs of solutions to multiple membrane problems (see for instance \cite{SY1, SY2, SY3, DSS}).

For our energy, when $F=0$, triple junctions were recently analyzed by \cite{OV2,CL3} (uniqueness of blow-ups and regularity of the free boundary). The difference in our case is that there are two possible blow-up configurations, one of them being unstable, which creates new challenges.

\subsection{Organization of the paper.}\label{subsec:org} 
The paper is organized as follows. Section \ref{sec:Basic} contains a discussion of basic properties of minimizers, such as optimal regularity, compactness properties, stability and non-degeneracy. Most technical proofs of these more standard results are postponed till the Appendix. In Section 3, we investigate the asymptotic behavior of global solutions with polynomial growth at
infinity and also provide the blow-up analysis at free boundary points where more than two components meet. In the symmetric case, $F=(1,1,1)$, the proof of Theorem \ref{mainintro} is obtained in Section \ref{sec:proof-symmetric} via an improvement of flatness technique. Section \ref{sec:PerturbativeCase} is devoted to the uniqueness of type I blow-ups and the regularity of the free boundary at
such points when $n=2, N=3$. In Section 6 we obtain the main results in the general non-symmetric case while the case $n\geq 3$ is analyzed in Section \ref{sec:higher}.
Finally, Section \ref{sec:Proof-Lemma-impro} contains the proof of a key technical lemma.

\section{Basic properties of minimizers}\label{sec:Basic}
This section is divided into four  different subsections devoted to optimal regularity, compactness, and uniqueness questions for minimizers of $J(\cdot, F)$. Some of the results were originally proved in \cite{CTV4}, however our potential does not satisfy all the assumptions there. Thus, for completeness we present the proofs in the Appendix for our setting.

We remark that the existence of minimizers of $J(\cdot, F)$ for a given boundary data $\Phi \in \mathcal A$ and the forcing term $F \in L^2$ follows from the direct method of the calculus of variations. Notice that since the set $\mathcal A$ of admissible functions is not convex, the uniqueness of minimizers is not guaranteed in general. 

\subsection{Optimal Regularity and Perturbative Estimates}\label{subsec:Optimal}

In what follows, for a given vectorial function $V=(v^1,\ldots, v^N)$, we denote by 
$$\hat V:=(\hat{v}^1,\ldots, \hat{v}^N), \quad \hat v^i := v^i  - \sum_{j \neq i} v^j.$$

Our first result is the derivation of the Euler-Lagrange equations associated to our minimization problem. Given that the proof is straightforward, we present it immediately rather than postponing it to the Appendix. 

\begin{lem}[Euler- Lagrange inequalities] 
Let $U \in \mathcal A$ minimize $J(\cdot, F)$ in $B_1$, with $F \in L^2(B_1)$.
Then, for all $i=1,\ldots , N$,
\begin{equation}\label{EL}\Delta u^i \geq f^i \chi_{\{u^i>0\}}, \quad \Delta \hat u^i \leq \hat f_U^i, \quad \text{in $B_1,$}\end{equation} where 
$$\hat f^i_U : = f^i\chi_{\{u^i>0\}} - \sum_{j\neq i}f^j\chi_{\{u^j>0\}}+ |f^i|\chi_{\{|U|=0\}}$$
\end{lem}
\begin{proof} Let $\varphi \geq 0$ be a smooth function with compact support in $B_1$. 
For a given $i\in \{1,\ldots, N\},$ let $V^{\eps,i}=(v^1,\ldots,v^i,\ldots v^N)$ with  
$$v^i:=(u^i - \eps \varphi)^+, \quad v^j=u^j \quad \text{for } j\neq i ,$$ 
where, for  clarity of exposition,  we have suppressed the dependence of the components on $\epsilon$.
Then $V^{\eps,i} \in \mathcal A$, and hence comparing $J(V^{\eps,i})$ and $J(U)$ we deduce that  
$$\int_{B_1} \frac 1 2 (|\nabla (u^i-\eps \varphi)|^2 -|\nabla u^i|^2) dx \geq  \int_{B_1} (f^i u^i  - f^i(u^i-\eps \varphi)^+) dx.$$ After dividing by $\eps,$ and restricting the integral on the right hand side on its support $\{u^i>0\},$ we get 
$$\frac 1 2 \eps \int_{B_1}|\nabla \varphi|^2 dx - \int_{B_1} \nabla u^i \cdot \nabla \varphi dx \geq \int_{B_1} f^i \varphi \chi_{\{u^i>0\}} dx - \int_{B_1\cap \{0<u^i<\eps \varphi\}}|f^i||\varphi| dx.$$ Here we have used that
$$(u^i-\eps\varphi)^+=u^i-\eps \varphi + (u^i - \eps \varphi)^-,$$
and that the term $(u^i - \eps \varphi)^-$ is bounded by $\eps\varphi$.

Letting $\eps \to 0$ the first inequality in \eqref{EL} follows.

For the second inequality, again fix  $i$ and $\varphi$ as above, we define $V^{\eps,i}=(v^1, \ldots, v^N) \in \mathcal A$ such that 
$$\hat v^i= \hat u^i + \eps \varphi,$$ 
that is  
$$v^i:=(\hat{u}^i + \eps \varphi)^+, \quad v^j:=(u^j-\eps \varphi)^+ \quad \text{for $j \neq i$}.$$

Let us notice that for any vectorial function $W \in \mathcal A$, and $i$ fixed
$$J(W)= \int_{B_1} (\frac 1 2 |\nabla \hat{w}^i|^2 +  F \cdot W)dx.$$ Therefore, comparing $J(U)$ with $J(V)$ we obtain the desired inequality  from similar computations as those above.

\end{proof}

We now state the optimal regularity of minimizers. For this purpose we need to assume that $F$ is bounded. 

\begin{thm}[Optimal regularity]\label{OR} Let $U \in \mathcal A$ minimize $J(\cdot, F)$ in $B_1$ and assume that $F \in L^\infty(B_1).$ Then, 
\begin{equation}\label{opt}\|U\|_{C^{0,1}(B_{1/2})} \leq C( \|U\|_{L^2(B_1)} + \|F\|_{L^\infty(B_1)}),
\end{equation}
for some dimensional constant  $C>0$.
\end{thm}

Next, we turn to the analysis of the perturbative setting, in which the free boundary is proven to be smooth and a free boundary condition is achieved in a classical sense. Precisely we have the following result.

\begin{prop}[Perturbative estimates]\label{pert}
Let $U \in \mathcal A $ minimize $J(\cdot, F)$ in $B_1$, $F\in L^\infty(B_1)$, satisfy $$\|U\|_{L^\infty(B_1)} \leq 1.$$ Assume that for some $\mu>0$
\begin{equation}\label{mu0}
|U(\frac 1 2 e_n)| \geq \mu .
\end{equation}
 Then, there exists $\eps>0$ depending on $\mu$ such that if 
\begin{equation}\label{flatnew} \Gamma(U) \subset \{|x_n| \leq \eps\}, \quad  \|F\|_{L^\infty(B_1)} \leq \eps,\end{equation} with $0\in \Gamma(U)$ then, up to relabeling the indices, 
$$u^i \equiv 0 \quad \text{in $B_{1/2}$, for all $i \geq 3$,}$$
and $\Gamma(U)$ is $C^{1,\alpha}$ in $B_{1/2}$, for all $0<\alpha <1,$ with $$|\nabla u^1| = |\nabla u^2| \quad \text{on $\Gamma(U)$.}$$
\end{prop}

\subsection{Sequences of minimizers.}\label{subsec:sequences}
Our compactness result reads as follows.

\begin{prop}[Compactness]\label{compact} Let $U_k \in \mathcal A$ be a sequence of minimizers for $J_k(\cdot):=J(\cdot, F_k)$ in $B_1$, with $F_k \in L^2(B_1).$ Assume $$U_k \to U \quad F_k \to F, \quad \text{in $L^2(B_1)$}$$ then $U$ is a minimizer for $J(\cdot, F).$  
In particular, if $U_k$ minimizes $J(\cdot, F)$ with $F \in L^\infty$ (that is $F_k=F$ for all $k$'s) and
$$M_k:=\|U_k\|_{L^2(B_1)} \to \infty,$$ then the normalized functions $ \frac{U_k}{M_k}$ converge (up to a subsequence) uniformly on compact sets to a minimizer of $J_0=J(\cdot, 0)$ in $B_1$. 
\end{prop}

In the case when the forcing terms are bounded, we prove a $H^1_{loc}$ convergence result of uniformly convergent sequences of minimizers.

\begin{lem}\label{H1loc} Let $U_k \in \mathcal A$ minimize $J_k(\cdot, F_k)$ in $B_1$ with $F_k \in L^\infty(B_1)$. Assume
$$U_k \to U, \quad F_k \to F \quad \text{uniformly on compact subsets of $B_1$}.$$ Then,
$$\nabla U_k \to \nabla U \quad \text{a.e. in $B_1$}.$$ In particular, 
$$U_k \to U \quad \text{in $H^1_{loc}(B_1).$}$$
\end{lem}

Finally, using compactness and monotonicity formulas we determine the asymptotic behavior of minimizers near a free boundary point.

\begin{prop}\label{prop:asympt}
Let $U \in \mathcal A$ minimize $J(\cdot, F)$ in $B_1$, for $F \in L^\infty(B_1)$. Assume $0 \in \Gamma(U)$. Then, as $x \to 0$, either $$|U(x)| = o(|x|)$$ or 
$$|U(x)| = a |x \cdot \nu| + o(|x|)$$ for some $a>0$ and $\nu \in \p B_1.$ In the latter case, the conclusion of Proposition \ref{pert} holds and $U \in C^{1,\alpha}$ in a neighborhood of $0.$
\end{prop}

\subsection{Stability and non-degeneracy}\label{subsec:stability}

In this subsection we examine the case when the forcing term $F \geq 0$ and obtain uniqueness, stability, and non-degeneracy.

Given $U,V \in \mathcal A$ we define  $d(U,V)$ to be the distance measured on the singular space given by the union of the positive coordinate axes: 
$$d(U,V): = \begin{cases} |u^i-v^i|, \quad \text{if $u^i,v^i>0$}\\
u^i+v^j \quad \text{if $u^i>0, v^j>0$}
\end{cases}.$$ Equivalently 
$$d(U,V)= \sum |u^i\pm v^j|\chi_{\{u^i>0\}} \chi_{\{v^j>0\}}$$
$$= \sum |u^i - \hat v^i|\chi_{\{u^i>0\}},
$$ with the minus sign when $i=j$ and plus sign when $i\neq j$. Observe that $d(U,V)= d(V,U)$.

\begin{lem}\label{duv}
Let $U,V \in \mathcal A$ be continuous and satisfy the Euler-Lagrange equations 
\begin{equation}\label{EL2}\Delta u^i \geq f^i \chi_{\{u^i>0\}}, \quad \Delta \hat u^i \leq \hat f_U^i, \quad \text{in $B_1,$}\end{equation}
\begin{equation}\label{EL3}\Delta v^i \geq f^i \chi_{\{v^i>0\}}, \quad \Delta \hat v^i \leq \hat f_V^i, \quad \text{in $B_1,$}\end{equation}
with $F\in L^\infty$ and $f^i \geq 0$. Then $d(U,V)$ is subharmonic in $B_1$.
\end{lem}

\begin{proof} It suffices to prove that $d(U,V)$ is subharmonic in its support.
Assume therefore $d(U(x_0),V(x_0))>0$.
By symmetry, we  may assume further that at $x_0$, $u^i >0$ and that $u^i>v^i$. Then, near $x_0$,  $$d(U,V)= u^i - \hat v^i$$ and $$\Delta (u^i - \hat v^i) \geq  f^i - \hat f_V^i \geq 0.$$\end{proof}

\begin{cor}[Stability]\label{stability} 
Let $U,V$ be as in Lemma \ref{duv}
and assume that $d(U,V) \leq \eps$ on $\p B_1$. Then  $d(U,V) \leq \eps$ in $B_1.$
In particular if  $U=V$ on $\p B_1,$ then $U=V$ in $B_1.$
\end{cor}

\begin{rem}\label{Appppp} The same proof shows that, if $V$ satisfies \eqref{EL3} with right hand side $H$, then 
\begin{equation}\label{app}
\Delta d(U,V) \geq - \|F-H\|_\infty.
\end{equation}\end{rem}

\begin{lem}[Non-degeneracy]\label{ND} Let $u^i$ be a continuous function satisfying  \begin{equation}\label{EL22}\Delta u^i \geq f^i \chi_{\{u^i>0\}}\end{equation} and 
assume that $f^i\geq \mu>0$. If $u^i(0) >0$, then $$\max_{\p B_r} u^i \geq c(\mu) r^2, \quad 0<r<1,$$ with $c$ depending on $\mu.$

\end{lem}
\begin{proof}Our claim follows from \eqref{EL22}, by comparing $u^i$ with the explicit barrier 
$$w(x): = \frac{\mu}{2n} |x|^2.$$ Indeed, if $$u^i \leq w \quad \text{on $\p B_r$,}$$ then, by the maximum principle applied in $\{u^i>w\}$, we find that 
$$u^i \leq w \quad \text{in $B_r$}.$$ This contradicts that $u^i(0)>0.$
 \end{proof}

By standard arguments using compactness, optimal regularity, and non-degeneracy, we conclude that the following lemma holds.

\begin{lem}[Free boundary convergence]\label{fbconv}Let $U_k \in \mathcal A$ be a sequence of minimizers to $J_k(\cdot):= J(\cdot, F_k)$ in $B_1$, with $F_k \in L^\infty(B_1)$, such that 
$$U_k \to U, \quad F_k\to F, \quad \text{uniformly on compact sets in $B_1$}.$$ Assume that $f^i>0, i=1,\ldots, N$. Then,
$$\p \{u^i_k>0\} \to \p \{u^i>0\},$$in the Hausdorff distance sense on compact sets in $B_1$.\end{lem}

\subsection{Monotonicity Formulas}\label{subsec:mono}
The following classical monotonicity formulas hold in our setting as well (see \cite{W, A}).

\begin{prop}[Weiss Monotonicity Formula]\label{Weiss} Let $U$ be a minimizer of $J(\cdot, F)$, in $B_1$, with $F$ constant. 
Then,
$$\Phi_U(r):= r^{-(n+2)} \int_{B_r} (\frac 1 2 |\nabla U|^2 + F \cdot U) dx -  r^{-(n+3)}\int_{\p B_r} |U|^2 d\sigma$$ 
is monotone increasing and if it is constant $U$ is homogeneous of degree 2.
\end{prop}

\begin{prop}[Almgren Frequency]\label{almgren}
Let $U$ be a minimizer of $J_0(\cdot):=J(\cdot, 0)$ in $B_1$. Then,
$$\gamma(r):= \frac{r\int_{B_r}|\nabla U|^2dx}{\int_{\p B_r}|U|^2}$$ is monotone increasing and if it is constant $U$ is homogeneous of degree $\gamma$.  \end{prop}

\section{Asymptotic behaviour of global solutions}\label{sec:Asympt-global} 
In this section we investigate the asymptotic behavior of global solutions with polynomial growth at infinity. This will lead to the proof of our main Theorem \ref{mainintro}, that is the classification of such solutions in two dimensions in the case when $N=3$. We also provide the blow-up analysis at free boundary points where more than two components meet.

\subsection{Blow-downs of global solutions.}\label{subsec:Blow-down}
Assume $U$ is a global minimizer of $J(\cdot, F)$, $F\in L^\infty$ with polynomial growth at infinity, $$|U(x)| = O(|x|^m), \quad \text{ for some $m \in \mathbb N$}.$$ 
Set
$$a_r:= \max_{B_r} |U|.$$
We start with the following lemma, in which we characterize the so-called blow-down limits when $U$ grows faster than quadratic.

\begin{lem}\label{Uobar} If 
\begin{equation}\label{case1}
    \lim_{r\to \infty} \frac{a_r}{r^2} = \infty.\end{equation} Then, there exists a sequence of radii $r_j \to \infty$ such that
$$U_j(x):=\frac{U(r_j x)}{a_{r_j}} \to \mathcal U_0, \quad \text{uniformly on compact sets},$$ with $\mathcal U_{0}$ a non-trivial $m_0$-homogeneous minimizer to $J_0:=J(\cdot, 0)$ with $m_0 \leq m$.

In particular, in a neighborhood of a point where  $\p \{|\mathcal U_0| >0\}$ is smooth, 
$\p \{|U_j| >0\}$ converges in $C^{1,\alpha}$ to $\p \{|\mathcal U_0| >0\}$.
\end{lem}

We call the blow-down limit $\mathcal U_0$ a type I blow-down solution. These are homogeneous solutions to the optimal partition problem for the operator $J_0$.

\begin{proof} Once we establish the existence of the $U_j$'s, the claim on the convergence of the free boundaries follows immediately from Propositions \ref{pert} and \ref{prop:asympt}.

Denote by $$b_{j}:=a_{2^{j}}, \quad b_j =O(2^{mj})$$ where the last inequality follows from our growth assumption.
Fix $\delta>0$ small. Then, it follows that there exists a sequence $j_k \to \infty, k\to \infty,$ such that, for every $l=0,1,\ldots, $ 
$$b_{j_k+l} \leq 2^{(m+\delta)l} b_{j_k},$$ for all large $k$'s. This is because the sets $$S_l:=\{k \in \mathbb{N} \ | \frac{b_{k+s}}{b_k} \leq 2^{s(m+\delta)}, \ s=1, \ldots, l, \}$$ are unbounded, in view of the growth of the sequence $b_j.$

Thus, the sequence $$U_{j_k}(x):=\frac{U(2^{j_k}x)}{b_{j_k}}$$
satisfies 
$$\|U_{j_k}\|_{L^\infty(B_1)} = 1, \quad |U_{j_k}(x)| \leq 2^{l(m+\delta)}, \quad \text{in $B_{2^l}$}$$ for all large $k$'s.

Moreover, as $k\to \infty$, since \eqref{case1} holds, $$F_{j_k}(x):= \frac{r^2}{a_r}F(r x) \to 0, \quad r=2^{j_k}.$$
Thus, we can extract a convergent subsequence $U_{j_k} \to \mathcal U$ uniformly on compact sets, which minimizes $J_0$ and 
$$\|\mathcal U\|_{L^\infty(B_1)} =1, \quad \mathcal U(x)= O(|x|^{m+\delta}) \quad \mbox{as $|x| \to \infty$}.$$

Now, by Almgren's frequency formula, the function
$$\gamma_{\mathcal U}(r):= \frac{r\int_{B_r}|\nabla \mathcal U|^2 dx}{\int_{\p B_r} |\mathcal U|^2 d\sigma}$$ is increasing and in view of the growth of $\mathcal U$ we get that 

$$\bar \gamma:=\lim_{r\to \infty} \gamma_\mathcal U(r) \leq m+\delta.$$
 Here we used that 
\begin{equation}\label{gammaUderivative}
\gamma_\mathcal U(r) = \frac r 2 \frac{d}{dr} \log \fint_{\p B_r} |\mathcal U|^2d \sigma .
\end{equation}

We claim that there exists a non-zero homogeneous of degree $m_0:=\bar\gamma$ minimizer $\mathcal U_0$ to $J_0$ such that, 

$$V_{R}(x):=\frac{\mathcal U(R x)}{\|\mathcal U\|_{L^\infty(B_{R})}} \to\mathcal U_0, \quad \text{on compact sets along a subsequence $R=R_l\to \infty.$}$$ Notice that
$$\gamma_{V_R}(r) = \gamma_\mathcal U(Rr) \leq m_0.$$ From here, using \eqref{gammaUderivative}, we have that $L^2$ norm of $V_R$ on $\p B_r$ is bounded in terms of $r$ and $m_0$ ($r \geq 1$).
We conclude that we can extract a convergent subsequence $V_{R_l}$ that converges uniformly on compact sets to a function $\mathcal U_0$ with
$$\gamma_{\mathcal U_0} (r) =m_0, \quad \|\mathcal U_0\|_{L^\infty(B_1)}=1.$$ Hence $\mathcal U_0$ is a non-zero homogeneous minimizer of degree $m_0.$

Therefore, given $\eps>0$ we can choose $R_l$ such that  $|V_{R_l} - \mathcal U_0| <\eps$ in $B_{1/\eps}$ and hence the rescaling $$\frac{U(r_j x)}{a_{r_j}}, \quad r_j:=2^{j_k}R_l$$ are well-approximated by $\mathcal U_0$ in $L^\infty$, which proves the lemma with $m_0 \leq m+\delta$. Now, taking a sequence of $\delta \to 0$ and using a diagonal subsequence argument, we obtain the conclusion for an $m_0\leq m,$ as in the last part of the proof.

\end{proof}

\begin{rem}\label{U0n2} In dimension  $n=2$, type I cones $\mathcal U_0$ are given by the rotations (around the origin)  $r^{k/2} \cos \frac{k}{2} \theta$, $k\in \mathbb{N}$ and the free boundaries are the rays $$\theta=\frac{2\pi i}{k}, i=0,1,\ldots, k-1.$$
This follows from the fact that $|\mathcal U_0|$ is harmonic in its positivity set, and by the Euler-Lagrange equation \eqref{EL} the gradients on either side of the rays $\{\mathcal U_0=0\}$ must be equal. 

We remark that $k$ must satisfy $k\geq 3$. Indeed, since $U_{j}(0) \to 0$, then $\mathcal U_0(0)=0$. Since $\mathcal U_0$ is non-trivial, then $k\neq 0.$ Moreover, since $\mathcal U_0$ is a minimizer to $J_0$, it has to be locally Lipschitz, thus $k\neq 1$. Finally, if $k=2$ then in view of Proposition \ref{pert}, $\hat U_j \to x_1$ in $C^{1,\alpha}$. Then, $$\nabla \hat U_j(0)= \frac{r_j}{a_j}\nabla \hat U(0) \to 1, $$ which contradicts \eqref{case1}.  \end{rem}

\begin{prop}\label{not1} Let $U \in \mathcal A$ be a global minimizer of $J(\cdot, F)$, $F\in L^\infty$, $F\geq 0$, with polynomial growth at infinity, $$|U(x)| = O(|x|^m),  \quad \text{for some $m \in \mathbb N$}.$$ Assume $n=2$, $N= 3$,$$\{u^i>0\} \quad  \text{are connected,}$$ and $$\p\{u^i > 0\} \cap \p\{u^j >0\} \neq \emptyset \quad \text{for all $i,j \in \{1,2,3\}$.}$$  Then $$\liminf_{r \to \infty} \frac{a_r}{r^2} < +\infty.$$
\end{prop}

 Proposition \ref{not1} will be proved in Section \ref{sec:PerturbativeCase}. At this point, we only show that if the conclusion of this proposition does not hold, then the only possible blow-down profiles are the 3/2-homogeneous solutions.

 \begin{lem}\label{time}Let $U$ be as in Proposition \ref{not1}. If $U$ satisfies \eqref{case1}, then
up to rotations the type I blow-down cones are given by 

$$|\mathcal U_0| = r^{3/2}|\cos \frac{3\theta}{2}|.$$
 \end{lem}

 \begin{proof} In view of the Remark \ref{U0n2},
it suffices to show that $|\mathcal U_0|:=r^{k/2}|\cos \frac k 2 \theta|$ with $k\geq 4$ cannot occur as a blow-down type I cone in the conclusion of Lemma \ref{Uobar}.

Assume by contradiction that such a $\mathcal  U_0$ is a blow-down. Then $\{|\mathcal U_0|>0\}$ consists of $k$ sectors of opening $\frac{2\pi}{k}$. Let $S_i^1, i=1, \ldots, l,$ denote the sectors where $u_0^1>0.$ Since ${\hat {u_0}}^1$ is superharmonic (see \eqref{EL}), these sectors cannot be adjacent to each other.
Given that $N=3$ and $k\geq 4$, we may assume without loss of generality that $l\geq 2$.

Let $x_i$ be the midpoint of the arc $S^1_i \cap \p B_1, i=1,2, \ldots l.$
From the uniform convergence of $U_j$ to $\mathcal U_0$, and the $C^{1,\alpha}$ convergence of the free boundaries (see Proposition \ref{pert}), we know that $x_i \in \{u_j^1>0\}$ and $\{u_j^1>0\} \cap \p B_1$ consists exactly of $l$ connected arcs. Since $\{u_j^1>0\}$ is connected, there is a polygonal path $P$ (with no self intersection) which joins two different $x_i$'s, fully contained in $\{u_j^1>0\}$. Moreover, by cutting off the polygonal line wherever it intersects the boundary of $B_1$, we may assume that $P$ connects two $x_i$'s (possibly different than the ones we started out with) and either lies completely inside $B_1$ or completely in $\bar B_1^c$ (except for the endpoints).

We claim that $P$ must lie in $B_1$. Indeed, say $P$ joins $x_1, x_2$, and assume by contradiction that it lies outside $B_1$. Then $P$ and the segments from the origin to $x_1$ and $x_2$ respectively, enclose a compact region $D$. Thus, there is at least a component of $U_j$ (different than $u_j^1$), say $u_j^2,$ such that the open set 
$$\mathcal O:=\{u_j^2>0\} \cap \bar{B}^c_1 \cap D$$ is non-empty. 
Notice that the maximum of $u_j^2$ in the compact set $\overline{\mathcal O}$ has to occur in the interior, which contradicts the fact that $u_j^2$ is subharmonic ($f_2 \geq 0$).
Indeed $u^2_j=0$ on the portion of the boundary $\p \{u_j^2>0\}\cap \p \overline{\mathcal{O}}$, while 
uniform convergence of $U_j$ to $\mathcal U_0$ in $B_2$ shows that the maximum cannot occur on $\p \overline{\mathcal O}\cap\p B_1$.

In conclusion, $P$ lies in $B_1$, which means that it disconnects $B_1$ into two connected components, $D_2,D_3$. We deduce that $\{u_j^2>0\} \cap B_1$ cannot intersect both $D_2$ and $D_3$ (and similarly for $\{u_j^3>0\} \cap B_1$). Otherwise, again by connectivity, we could join points $y_2 \in D_2, y_3\in D_3$ with a polygonal path in $\{u_j^2>0\}$ which exits $B_1$ and reach a contradiction as above with the roles of $u_j^1$ and $u_j^2$ reversed. On the other hand, since $S^1_1$ and $S^1_2$ are not adjacent, $D_2$ (and $D_3$) must intersect either $\{u^2_j>0\}$ or $\{u^3_j>0\}.$

Thus, say, 
$$\{u_j^2>0\} \cap B_1 \subset D_2 \quad \text{and} \quad \{u_j^3>0\} \cap B_1 \subset D_3.$$ 
This implies that in $B_1$
$$\p \{u_j^2>0\} \cap \p \{u_j^3>0\} = \emptyset$$
which means that 
$$\p \{u^2>0\} \cap \p \{u^3>0\} \cap B_{r_j}= \emptyset$$
for a sequence of $r_j\to \infty$, contradicting our assumptions.

\end{proof}

B the monotonicity lemma,  Proposition \ref{Weiss}, we can show that when $U$ has a quadratic bound along a sequence of radii and $F$ is constant, then the quadratic rescalings converge to a 2-homogeneous minimizer of $J(\cdot, F).$
\begin{lem}\label{blowup}Let $U \in \mathcal A$ be a global minimizer to $J(\cdot, F)$, $F$ constant. Assume that \begin{equation}\label{limit} \liminf_{r\to\infty} \frac{a_r}{r^2} < \infty.\end{equation}
Then, there exists $r_j \to \infty$, $j \to \infty$, such that 
$$U_{j}(x):=\frac{U(r_jx)}{r_j^2} \to U^F_0(x),$$ uniformly on compact sets, where $U_0^F$ is a minimizer to $J(\cdot, F)$ which is homogeneous of degree 2.
\end{lem}
\begin{proof} The convergence on compact subsets of $B_1$ of the rescalings of $U$ along a subsequence is trivial. We only need to show that the limit $U_0^F$ is 2-homogeneous. Then up to a further rescaling, the claim follows.

By Caccioppoli's inequality (see Lemma \ref{lem:Caccioppoli}), we have 
$$\Phi_{U}(1/2) \leq C(\|U\|^2_{L^\infty(B_1)} + \|F\|^2_{L^\infty(B_1)}),$$ which after rescaling gives,
$$\Phi_{U}(\frac r 2) \leq C((\frac{a_r}{r^2})^2 + \|F\|^2_{L^\infty}).$$ Thus, since $\Phi_U$ is monotone, we obtain that it is bounded above and has a finite limit $\Phi_U(+\infty).$ Therefore, using that $U_{r_j}$ converges uniformly to $U_0^F$ on compact subsets of $B_1$, together with Lemma \ref{H1loc} we get
$$\Phi_{ U}(r_j\rho) = \Phi_{U_{j}}(\rho) \to \Phi_{U_0}(\rho).$$
This gives that $\Phi_{U_0^F}(\rho)= \Phi_U(+\infty)$ is constant, and $U_0^F$ must be 2-homogeneous. \end{proof}

As mentioned earlier, we refer to the blow-down function $U_0^F$ as a type II  cone. Notice these are 2-homogeneous minimizer of $J(\cdot, F)$, for $F$ constant. For the sake of accuracy,  we  define different type  of global solutions more precisely. 

\begin{defn}\label{typeI} 
A homogeneous global minimizer $\mathcal U_0$ of $J_0:=J(\cdot, 0)$ is called a type I cone.\end{defn}
\begin{defn}\label{typeII}
A  2-homogeneous global minimizer $U_0^F$ of $J(\cdot, F)$, $F$ constant, is called a type II cone. We omit the dependence on $F$ when $F=(1,1,1).$
\end{defn}

\begin{rem}\label{all}
The result above can be stated in a stronger sense, that is any sequence $r_j \to \infty$ has a subsequence along which the conclusion of the lemma holds. This is a consequence of the fact that assumption \eqref{limit} is only used to show that $\Phi_U(+\infty)$ is bounded. Equivalently, this means that the rescalings $\frac{U(rx)}{r^2}$ for all large $r$'s, are uniformly approximated in $B_1$  by a 2-homogeneous solution $U_0^F$ that depends on $r$.
\end{rem}

Finally, in view of Proposition \ref{not1}, the conclusion of Lemma \ref{blowup} always holds for a minimizer $U$ as in Theorem \ref{mainintro}. This leads us to investigate 2-homogeneous solutions of $J(\cdot, F)$ in 2d, when $F$ is constant and positive, and $N=3$.

\subsection{Global 2-homogeneous solutions (type II) in 2d.}\label{subsec:global}
Let $n=2$. Assume that $$\text{$F > 0$ is constant}$$ and let $U_0^F$ be a global 2-homogeneous minimizer of $J(\cdot, F).$ Then, at most three components of $U_0^F$ are non-zero. Precisely, up to rotations and relabeling of the components, one of the following cases occurs. For notational simplicity we drop the dependence on $F$ in the components. Minimality of these cases below follows by the uniqueness of solutions that satisfy the Euler-Lagrange equation, see Corollary \ref{stability}.

{\it Case 1: One non-trivial component.}$$u_0^1 = \frac 12 f^1 (x_2^+)^2; \quad \text{or} \quad 
u_0^1= \frac 1 2 f^1 (a x_1^2 + (1-a) x_2^2), \quad a \in [0,1];$$

{\it Case 2: Two non-trivial components.} $$u_0^1= \frac 12 f^1 (x_2^+)^2, \quad  u_0^2= \frac 12 f^2 (x_2^-)^2;$$

{\it Case 3: Three non-trivial components.} 

For $i=1,2,3$, $u_0^i$ is the 2-homogeneous solution to $\Delta u_0^i=f^i$ in an angle of opening $\theta_i \in [\frac \pi 2, \pi]$ which vanishes on the boundary. The $\theta_i$ are uniquely determined by the compatibility equations
 \begin{equation}\label{uof}
 f^1\tan\theta_1=f^2\tan\theta_2=f^3\tan\theta_3, \quad \sum \theta_i=2\pi, \quad \theta_i \in [\frac \pi 2, \pi].
 \end{equation}
 
 Indeed, up to rotations, in polar coordinates
 $$u_0^1 = \frac 1 4 f^1(r^2 - \frac{1}{\cos \theta_1} r^2\cos(2\theta)), \quad |\theta|<\frac{\theta_1}{2},$$
and
$$|\nabla u_0^1| = \frac 1 2 f^1 r |\tan \theta_1|, \quad \text{on $|\theta|=\frac{\theta_1}{2}$,}$$ which by Proposition \ref{pert} leads to \eqref{uof}.
The uniqueness of the $\theta_i$'s in \eqref{uof} comes from the fact that $\tan \theta$ is increasing on the interval $(\pi/2,\pi]$, hence for each $s\le 0$ we can solve $f^i \tan \theta_i=s$ uniquely with $\theta_i$ depending monotonically on $s$.

The solution $U_0^F$ is unique modulo rotations and reflections depending on the order of the components in the counterclockwise direction.

\subsection{General global solutions $U_t$ in 2d, when $N=3$.}\label{subsec:general-global}
Next, we investigate 2d global solutions that are not necessarily homogeneous, when $N=3$. We are interested in solutions which are asymptotic to $U_0^F$ at infinity and to $\mathcal U_0$ near 0. 
First, we consider the symmetric/isotropic  case $$F=(1,1,1),$$ 
since in this case the construction is simpler. In Section \ref{sec:non-sym-F}, we will provide the construction in the general non-symmetric case. Here, for notational simplicity, we drop the dependence on $F$ in our notation.

Let $U_0$ be the homogeneous of degree 2 solution in 2d for $N=3$ domains, given by 
\begin{equation}\label{U0}u_0^1:= \frac 1 4 (3x^2 -y^2)=\frac 1 4 r^2(1+2\cos 2\theta) \quad \text{in $-\pi/3 \leq \theta \leq \pi/3$}\end{equation} and $u_0^2, u_0^3$ the even reflection of $u_0^1$ across $\theta=\pm \pi/3.$ 

For $t\geq 0$, define $U_t$ as the sum of $U_0$ and t times the homogeneous of degree 3/2 solution for $J_0$. Precisely,  \begin{equation}\label{36}u^1_t:=u_0^1 + t \cdot r^{3/2}\cos \frac{3}{2}\theta, \quad \text{in $-\pi/3 \leq \theta \leq \pi/3$}.\end{equation} Similarly, for $t<0,$ there is a solution $U_t$ such that for $r$ large, tending to $\infty,$ it has an expansion \begin{equation}\label{exp}u^1_t:= u_0^1 + t \cdot r^{3/2}\cos \frac{3}{2}\theta + O(|t|^7r^{-3/2}), \quad \text{in $-\pi/3 \leq \theta \leq \pi/3$}.\end{equation} 

Precisely, let us construct $v:=u^1_{-1}$ by solving first the obstacle problem
$$\Delta v_R = 1\chi_{\{v>0\}}, \quad \text{in $\{|\theta|<\frac \pi 3\} \cap B_R,$ }$$ with $R$ large 
and boundary data 
$$v_R=0 \quad \text{on $|\theta|=\frac \pi 3$}$$
$$v_R = u_0^1 - R^{3/2}\cos \frac{3\theta}{2}, \quad \text{on $\p B_R \cap \{|\theta|<\frac{\pi}{3}\}$}.$$
Then define $v$ as the limit of $v_R$ as $R \to \infty.$ Indeed, by the maximum principle, 
\begin{equation}\label{expansionvR}u_0^1 - |x|^{3/2}\cos \frac 3 2 \theta \leq v_R \leq u_0^1 + C |x|^{-3/2}\cos \frac 3 2 \theta - |x|^{3/2}\cos \frac 3 2 \theta,\end{equation} for a large constant $C>0,$ 
which justifies both the convergence of the $v_R$'s and the expansion in \eqref{exp}. Notice that again by the maximum principle, 
\begin{equation}\label{less}v<u_0^1 \quad  \text{in $|\theta|<\pi/3$}.\end{equation}

We claim that $v$ vanishes in a neighborhood of the origin, hence $\Gamma(U_{-1})$ does not have a triple junction at $0.$
Otherwise, $0 \in \Gamma(U_{-1})$ and $|U_{-1}|$ grows at a quadratic rate from the origin. The lower bound is a consequence of non-degeneracy, while the upper bound follows from \eqref{less}. Now, by Weiss monotonicity formula we conclude that $U_{-1}$ has a homogeneous of degree 2 blow-up profile $W.$ Since $W \neq 0$ by the quadratic behaviour,  in view of the classification of 2-homogeneous solutions, $W = U_0$, which means 
$$\Phi_{U_{-1}}(0+)= \Phi_W = \Phi_{U_0 }= \Phi_{U_{-1}}(+\infty)$$
where the last equality follows from the expansion \eqref{exp} with $t=-1$. This gives that $U_{-1}$ is 2-homogeneous, which is a contradiction.

We remark that in the interior of the angle $|\theta|< \pi/3$, the free boundary $\Gamma(U_{-1})=\p \{v>0\}$ is a smooth curve that becomes tangent to the rays $\theta= \pm \pi/3$ at the end points. This is because the directional derivatives $\p_e v \ge 0$ when $e$ lies in the interior of the angle $|\theta|< \pi/3$, which follows by the maximum principle and expansion \eqref{expansionvR}. Thus the free boundary of $v$ is a Lipschitz graph in its unconstrained part. 

Finally, for $t<0$, we define $U_t$ by quadratic rescaling, $$u^1_t:= t^4 u^1_{-1}(x/t^2),$$ so that the expansion above holds. Moreover, notice that $U_t$ is differentiable in $t$ and 
\begin{equation}\label{ptut}
\p_t u^1_t|_{t=0} = r^{3/2}\cos \frac{3\theta}{2}.
\end{equation}

\begin{rem}\label{rem38}Similarly, one can construct global solutions with $N \ge 5$ phases and growth $O(|x|^{N/2})$ as $|x| \to \infty$ which have an asymptotic expansion
$$|U|=|x|^{N/2} |\cos ( \frac N 2 \theta)|+ |x|^2 g_2(\theta) + O(|x|^{-N/2}) \quad \mbox{as} \quad |x| \to \infty.$$

\end{rem}

\subsection{Blow-up analysis at triple junctions}\label{subsec:Blowup}
Similarly to our blow-down analysis, we can perform a blow-up analysis at a free boundary point by rescaling  a minimizer $U$ around such a point according to its growth. Precisely, if $U$ minimizes $J(\cdot, F)$ in $B_1$ with $F$ constant, and $0 \in \p \{|U|>0\}$, then either
\begin{equation} \label{c1}
    \lim_{r\to0} \frac{a_r}{r^2} = +\infty, \quad \quad \quad a_r:= \max_{B_r} |U|,
\end{equation}
and up to a subsequence $$U_r(x):= a_r^{-1} U(rx) \to \mathcal U_0 \quad \text{a type I cone}$$
or 
\begin{equation} \label{c2}
\liminf_{r\to0} \frac{a_r}{r^2} < +\infty
\end{equation}
and up to a subsequence $$U_r(x):= r^{-2} U(rx) \to U^F_0 \quad \text{a type II cone}.$$

Next we show when $n=2$, $F>0$ constant, and the number of components $N$ is arbitrary, that at most three components can meet at a free boundary point. We refer to such a point where precisely three components meet as a {\it triple junction}. Precisely, we obtain Theorem \ref{Pintro}, which we restate for convenience.
\begin{thm} Assume $n=2$, $F>0$ constant, $N \ge 3$, and that $0$ belongs to the free boundaries of at least three components. Then there are exactly three nontrivial components in a neighborhood of $0$. Moreover, there are only two possible blow-ups (up to rotations and reflections) at $0$ 
$$|\mathcal U_0|=r^{3/2} |\cos \frac{3 \theta}{2}| \quad \mbox{(type I), \hspace{5mm} or} \quad \quad U_0^F \quad \mbox{as in \eqref{uof} (type II)}.$$
\end{thm}

\begin{proof} If \eqref{c2} holds, then the blow-up must be $U_0^F$ by the  monotonicity formula,  Proposition \ref{Weiss},  and the classification of $2$-homogeneous solutions in Section \ref{subsec:global}. 

If \eqref{c1} holds, then as in the proof of Lemma \ref{Uobar}, we end up with a type I blow-up homogeneous solution $\mathcal U_0$ of homogeneity at most $2+\delta$, for some $\delta >0$ as small as we wish. This means that either $\mathcal U_0$ is $3/2$-homogeneous as above, or it is $2$-homogeneous i.e. 
$$\mathcal U_0= c|xy|, \qquad \hbox{for some } c> 0.$$

We show that this last possibility cannot occur. Otherwise, there is a sequence of annuli $B_{2r_k} \setminus B_{r_k}$, $r_k \to 0$, on which $|U|>0$ consists of 4 connected components. This follows from the fact that $U_{r_j} \to \mathcal U_0$ uniformly in $B_1 \setminus B_{1/2}$ and the free boundaries $\Gamma(U_i) \to \Gamma(\mathcal U_0)$ in the $C^{1,\alpha}$ sense due to Proposition \ref{pert}.
As in the proof of Lemma \ref{time} we can infer that in a neighborhood of the origin, the set $\{|U|>0\}$ consists of precisely 4 connected components that we denote by $\Omega_i$, $1\le i \le 4$, such that $\p \Omega_1 \cap \p \Omega_3=\{0\}$, and $\p \Omega_2 \cap \p \Omega_4=\{0\}$. Then the function
$$u:=|U| \chi_{\Omega_1 \cup \Omega_3} - |U| \chi_{\Omega_2 \cup \Omega_4},$$
has Laplacian that is constant in its connected components of the positivity and negativity sets, and satisfies $u \triangle u \ge 0$. By the results in \cite{AS}, we obtain $u \in C_{loc}^{1,1}$ and we contradict \eqref{c1}. 

We give a short explanation of the result in \cite{AS} and how the $C^{1,1}$ regularity of $u$ is obtained at the origin. Since the Laplacian has a positive jump from $\{u>0\}$ to $\{u<0\}$ then the absolute value of the directional derivatives $\partial _e u$ are subharmonic (see for example Lemma \ref{impro} below). Then one can apply the ACF monotonicity formula for $(\partial_e u)^+$ and $(\partial_e u)^-$ and find that they can grow at most linearly near the origin. 
 
\end{proof}

\section{The proof of Theorem \ref{mainintro} in the symmetric case}\label{sec:proof-symmetric}

 Our main purpose here is to prove Theorem \ref{mainintro} in the symmetric case $F=(1,1,1)$ for which the family $U_t$ has already been constructed. We will use an improvement of flatness technique, thus we start in the following perturbative setting.

Let $U$ be a minimizer of $J(\cdot, F)$ in $B_1$ and let $U_0$ be the 2-homogeneous solution defined in \eqref{U0}. If 
\begin{equation}\label{eq:U-ep}
    d(U,U_0) \leq \eps  \hbox{ in }  B_1, 
\end{equation}
 then 
$$
\Gamma(U) \subset \{|U_0| <\eps\} \cap B_1.$$ 
In particular, by non-degeneracy, in the annulus $B_{1-\delta} \setminus B_\delta$,  \begin{equation}\label{UF}\text{$\Gamma(U)$ is included in a $C\frac \eps \delta$ neighborhood of the rays $\theta=\pi, \theta=\pm\frac \pi 3$}\end{equation} and  near the rays only the free boundaries of two components meet, i.e.
$$\Gamma(U)=\p\{u^1>0\} \cap \p \{u^2>0\} \quad \text{near $\theta = \frac \pi 3$}$$
and similarly for the other two rays.
Moreover, we know that near the bisectors of the three sectors, 
$$u^1>0 \quad \text{on $\theta=0$}, \quad u^2>0 \quad \text{on $\theta=\frac 2 3 \pi$}, \quad u^3>0 \quad \text{on $\theta=\frac 4 3 \pi$}.$$
Hence, 
\begin{equation}\label{double}\hat u^1 = - \hat u^2 \quad \Longrightarrow \quad \Delta \hat u^1=\hat f^1_U \quad \text{in $\{0 < \theta < \frac 2 3 \pi\} \cap B_{1-\delta} \setminus B_\delta$}\end{equation} and similarly for the other two sectors.

To a solution $U$ as above, we associate the double-valued function $$u_1 \to -u_2 \to u_3 \to -u_1 \to u_2 \to -u_3$$ as we move counterclockwise from one sector to the next in the annulus $B_{1-\delta} \setminus B_\delta,$ and we circle the annulus twice. By abuse of notation, we denote such a double-valued function still by $U$, which satisfies 
$$|\Delta U| \leq 1$$ and therefore
is of class $C^{1,\alpha}$. Equivalently, $U$ can be identified by a single-valued odd function $h$ defined on the same  annulus,
which assigns to the point of polar coordinates $(r,\theta)$ the value of $U$ at $(r, 2\theta)$, $\theta \in [0,2\pi],$ that is
\begin{equation}\label{Uh}U(re^{2i\theta}) = h(re^{i\theta}), \quad \text{$h$ odd}\quad \theta \in [0, 2\pi].\end{equation}

\begin{defn}
We define the space $\mathcal D$ of double-valued functions on an annulus as those that can be represented by a continuous function $h$ as in \eqref{Uh}.\end{defn}
Notice that $\mathcal D$ is a vector space and derivatives of smooth functions in $\mathcal D$ belong to $\mathcal D$ as well.

With this convention, and that \eqref{eq:U-ep} holds, set 
\begin{equation}\label{eq:V-ep}
    V_\eps: = \frac 1 \eps (U-U_0).
\end{equation}
Then,
$$|V_\eps| \leq 1.$$ We use the next general lemma that holds in arbitrary dimensions (proved in Section \ref{sec:Proof-Lemma-impro}) to understand the limit of $V_\eps$ as $\eps \to 0$, leading to the linearized equation which is of transmission type along the zero set of the original solution.

\begin{lem}\label{impro} Let $u_0 \in C^1(B_1)$ with $|\nabla u_0| \neq 0$, satisfy
$$\Delta u_0 = a^+\chi_{\{u_0>0\}} -a^- \chi_{\{u_0<0\}} \quad \text{in $B_1$},$$ with $a^\pm \in \R.$ Assume $u$ satisfies the same equation,
$$\Delta u = a^+\chi_{\{u>0\}} - a^-\chi_{\{u<0\}} \quad \text{in $B_1$},$$
and
$$|u-u_0| \leq \eps \quad \text{in $B_1$}.$$ Set 
$$v_\eps:= \frac{u-u_0}{\eps},$$ then as $\eps \to 0$, $v_\eps \to v$ uniformly on compact subsets of $B_1$ (up to a subsequence), and $v$ satisfies the transmission problem 
$$\Delta v=0 \quad \text{in $\{u_0\neq 0\}$}$$
$$v_{\nu^+} + v_{\nu^-} =\frac{a^++a^-}{|\nabla u_0|} v \quad \text{on $\{u_0=0\}$}.$$ Here $\nu^\pm$ are the unit normals to $\{u_0=0\}$ pointing to $\{u_0>0\}$ and $\{u_0<0\}$ respectively.
\end{lem}

In fact, Lemma \ref{impro} holds in slightly greater generality, in the sense that it applies for sequences $u_k$ and $u_{0,k}$ that converge to $u_0$. For completeness, we state it here.

\begin{lem}\label{improgeneral} Let $u_{0,k}, u_k \in C^1(B_1)$  satisfy
$$\Delta u_{k} = a_k^+\chi_{\{u_k>0\}} -a_k^- \chi_{\{u_k<0\}} \quad \text{in $B_1$},$$ 
$$\Delta u_{0,k} = a_k^+\chi_{\{u_{0,k}>0\}} -a_k^- \chi_{\{u_{0,k}<0\}} \quad \text{in $B_1$},$$
with $a_k^\pm \in \R.$ Assume that $u_{0,k} \to u_0$ uniformly on compact sets, and $a^\pm_k \to a^\pm,$ with $u_0$ satisfying the hypotheses of Lemma \ref{impro} and 
$$|u_k-u_{0,k}| \leq \eps_k \quad \text{in $B_1$}.$$ 
Then 
$$v_{\eps_k}:= \frac{u_k-u_{0,k}}{\eps_k},$$ as $\eps_k \to 0$, $v_{\eps_k} \to v$ uniformly on compact subsets of $B_1$ (up to a subsequence), and $v$ satisfies the transmission problem 
in Lemma \ref{impro}.
\end{lem}

Applying the lemma above to $V_\eps$ from \eqref{eq:V-ep}, and letting $\delta \to 0$,  we conclude that $V_\eps$ converges (up to a subsequence) uniformly on compact subsets of $B_1 \setminus \{0\}$ to a limiting function $V \in \mathcal D$ that satisfies

\begin{align}\label{FBlinear}
& \Delta V=0 & \text{in $\{U_0\neq 0\}$}\\
\label{transmission} & V_{\nu^+} + V_{\nu^-} =\frac{2}{|\nabla U_0|} V & \text{on $\{U_0=0\}.$}
\end{align}

The transmission conditions \eqref{transmission} are imposed on the three open rays. 

\begin{rem}\label{linearsmooth}We remark that solutions in a neighborhood of a ray are smooth up to the ray, on either side of the ray. This is because solutions to a general transmission problem (in $\R^n$) across a hyperplane ($x=(x',x_n)$), 
\begin{align*}
& \Delta v = 0 \quad \text{in $B_1 \cap \{x_n\neq 0\}$}\\
& v_n^++v_n^-=h(x')v \quad \text{on $B_1 \cap \{x_n=0\}$}\end{align*}
can be reduced to the Neumann problem in a half-space for the function $$\frac 1 2 v(x', x_n) + \frac 1 2 v(x', -x_n).$$ In particular, if $h$ is smooth then solutions are smooth on both sides of the hyperplane.

\end{rem}

The next lemma shows that these are sufficient to determine the behaviour of a solution at the origin.

\begin{lem}\label{extend} Let $V \in \mathcal D$ be a bounded double-valued classical solution in $B_1\setminus \{0\}$ to \begin{align}\label{FBlinear2}
& \Delta V=0 & \text{in $\{U_0\neq 0\}$}\\
\label{b}& V_{\nu^+} + V_{\nu^-} \geq 0 & \text{on $\{U_0=0\} \cap \{V \geq 0\}.$}
\end{align} 

Then 
\begin{equation}\label{|V|}
\text{$|V|$ is subharmonic in $B_1$,}\end{equation}
\begin{equation}\label{V0}|V(x)| \leq C |x|^\alpha \|V\|_\infty,\end{equation}
and hence $V$ is continuous at zero with a H\"older modulus of continuity.
In particular, any solution to \eqref{FBlinear}-\eqref{transmission} is uniquely determined by its boundary values on $\p B_1.$ 
\end{lem}

\begin{proof} We first show the following decay of the $L^\infty$ norm in an annulus $B_{2r} \setminus B_{r/4}.$

Assume $V$ satisfies \eqref{FBlinear2}-\eqref{b},	
then \begin{equation}\label{osc}
\|V\|_{L^\infty(B_r \setminus B_{r/2})} \leq (1-c_0) \|V\|_{L^\infty(B_{2r} \setminus B_{r/4})},
\end{equation}for a dimensional constant $0<c_0<1$.
Notice that
$$V^+ \quad \text{is subharmonic in $B_1 \setminus \{0\}.$}$$
This is obviously satisfied in $\{U_0 \neq 0\},$ and on the rays due to the boundary condition \eqref{b}. On the other hand, $\{V^+=0\}$ covers at least half the measure of the double annulus. Applying the weak Harnarck inequality we obtain the claim \eqref{osc}. In particular, by varying $r$ in the estimate \eqref{osc}, we obtain the decay in full balls,
\begin{equation*}\label{oscnew}
\|V\|_{L^\infty(B_r)} \leq (1-c_0) \|V\|_{L^\infty(B_{2r})}.
\end{equation*}
from which the estimate claimed in the lemma follows (using only \eqref{b}).

For the uniqueness claim, assume $V$ and $W$ are two solutions which agree on the boundary of $B_1$. Then, $V-W$ vanishes continuously on $\p B_1 \cup \{0\}$ (by \eqref{V0}). If $V-W\ne 0$, then it achieves an interior maximum, which is a contradiction since our transmission problem satisfies the strong maximum principle in $B_1 \setminus \{0\}$. 

\end{proof}

In view of the uniqueness proved in the lemma above solutions to \eqref{FBlinear}-\eqref{transmission} can be obtained by minimizing the energy 
$$E(V):=\int |\nabla V|^2 dx + \int_{\{U_0=0\}} \frac{2}{|\nabla U_0|}V^2  d\sigma$$ among all double-valued functions $V$ for which $E(V) <\infty.$

Now, we obtain an expansion of a bounded solution to \eqref{FBlinear}-\eqref{transmission}.

\begin{lem}[Expansion in terms of homogeneous solutions $U_0$]\label{expandV} Let $V \in \mathcal D$ be a solution to \eqref{FBlinear}-\eqref{transmission} with $\|V\|_\infty \leq 1.$ Then 
$$V(x)={\bf m} \cdot \nabla U_0 + p r^{3/2} \cos\frac{3\theta}{2} + q \p_\theta U_0 + O(|x|^{5/2}),$$
where $|{\bf m}|, |p|$ and $|q|$ are universally bounded.
\end{lem} 
\begin{proof} We look for homogeneous solutions to \eqref{FBlinear}, 
$r^\lambda g(\theta),$ with \begin{equation}\label{g}\text{$g$ periodic of period $4\pi$ and $g(\theta+2\pi) = - g(\theta)$.}\end{equation} 
Then,
\begin{equation}
g'' + \lambda^2 g=0 \quad \text{except for $\theta_k=(2k+1)\frac{\pi}{3},$ $k \in \mathbb Z$}
\end{equation}
and at these points $g$ satisfies the jump condition,
\begin{equation}\label{g'}
(g')^+ - (g')^- = \frac{4}{\sqrt 3} g,
\end{equation} where $(g')^\pm$ denote the limits of $g'$ from the right and  the left respectively. 
Hence, on each interval $[\theta_k, \theta_{k+1}]$, $g=g_k$ has the form
\begin{equation}\label{gk}
g_k(\theta) = a_k\cos(\lambda \theta) + b_k \sin(\lambda \theta)
\end{equation}
and the jump condition implies that $g_{k+1}$ is obtained from $g_k$ from the formula, 
\begin{equation}
g_{k+1} (\theta) = g_k (\theta) + \frac{4}{\lambda \sqrt 3} g_{k}(\theta_{k+1}) \sin(\lambda(\theta-\theta_{k+1})).
\end{equation}
The eigenvalues and corresponding eigenfunctions $\lambda^2_m, g_m$ can be obtained by the standard min-max of the Rayleigh quotient 
\begin{equation}\label{quotient}
    \min \max_{g\in V_m}\frac{\int_0^{4\pi} (g')^2d\theta + 4/\sqrt 3 \sum_{0}^{5}g^2(\theta_k)}{\int_0^{4\pi} g^2d\theta}\end{equation}
where $V_m$ is a $m$-dimensional subspace of the $L^2$ space of functions satisfying \eqref{g}. Notice that there is a one-to-one correspondence  between the homogeneity $\lambda$ and the eigenvalue $\lambda^2.$

The analogous approach without the extra addend involving the values at the jumps, produces the standard homogeneities $1/2, 3/2, 5/2, \ldots$ of multiplicity 2, with corresponding eigenfunctions
\begin{equation}\label{zeroeigen}\cos \frac{\theta}{2},\sin \frac{\theta}{2}, \cos \frac{3\theta}{2}, \sin \frac{3\theta}{2},  \ldots\end{equation}
The eigenvalues for our problem are larger than these ones,
thus in particular, the fifth homogeneity (counted with multiplicity) is larger than 5/2. On the other hand the first four homogeneities are explicit, and they are $1,1, 3/2$ and $2$, and they correspond to 
$$\p_{x_1}U_0, \p_{x_2}U_0, r^{3/2} \cos\frac{3\theta}{2}, \p_\theta U_0,$$
restricted to the unit circle.

Now the result follows by standard Fourier decomposition. Precisely, given a continuous boundary data $g(\theta)$ on $\p B_1$ satisfying \eqref{g}, we can write the Fourier expansion relative to the orthonormal basis of $L^2(\p B_1)$ given by the eigenfunctions $g_m$'s, $$g= \sum_m a_m g_m(\theta) \quad \text{in $L^2(\p B_1)$}.$$
Then, solution to \eqref{FBlinear} with boundary data $g$ is given by 
$$V= \sum_m a_m r^{\lambda_m}g_m(\theta) \quad \text{in $L^2(B_1).$}$$
where we deduced that the $L^2$ convergence of the boundary data implies the $L^2$ convergence in the interior by \eqref{|V|}. The result follows since
$$|a_m| \leq \|g\|_2, \quad \|g_m\|_\infty \leq C \quad \text{by \eqref{gk}},$$ for $C$ universal.
\end{proof}

We denote by 
 $$V_{{\bf m},p, q}(x):={\bf m} \cdot \nabla U_0  + p r^{3/2} \cos\frac{3\theta}{2} + q \p_\theta U_0$$
 and by abuse of terminology  we refer to such functions as polynomials.
 Lemma \ref{expandV} then immediately gives the following.

 \begin{lem}\label{tt}\label{expandV1} Let $V \in \mathcal D$ be a solution to \eqref{FBlinear}-\eqref{transmission} with $\|V\|_\infty \leq 1.$ There exist $\rho>0$ small universal and a polynomial $V_{{\bf m}, p, q}$ with universally bounded coefficients, such that 
\begin{equation} |V- V_{{\bf m},p,q}|\leq \frac 1 8 \rho^2 \quad \text{in $B_{2\rho}$.}\end{equation}
\end{lem}

 We are finally ready to state and prove our improvement of flatness result.

\begin{prop}[Improvement of Flatness near type II cones]\label{IF} There exists $\eps_0, \delta, \rho$ universal, such that if $U$ minimizes $J(\cdot, F)$, $F=(1,1,1)$ in $B_1$ and for $\eps \leq \eps_0$
$$d(U,U_t) \leq \eps, \quad |t| \leq \delta, \quad \text{in $B_1$}$$
then $$d(U,\tilde U_{t+\eps p}) \leq \frac \eps 2 \rho^2 \quad \text{in $B_\rho(\eps \bar {\bf m})$},$$ with $\tilde U_{t+\eps p}$ a rotation by $\eps q$ and translation by $\eps \bar {\bf m}$ of $U_{t+\eps p},$ with $|\bar{\bf m}|, |p|$ and $|q| \leq C.$
 \end{prop}

\begin{proof} We argue by compactness. Assume the conclusion does not hold for a sequence of $t\to 0$ and a sequence of $\eps \to 0.$ Then, by Lemma \ref{improgeneral} the rescaled double-valued function $$V_\eps:= \frac{U - U_t}{\eps}$$ converges uniformly on compact subsets of $B_1 \setminus \{0\}$, to a solution $V$ of the linearized equation \eqref{FBlinear}-\eqref{transmission}, with $\|V\|_\infty \leq 1.$ We now apply Lemma \ref{tt} and conclude that

\begin{equation} |V_\eps- V_{{\bf m},p,q}|\leq \frac 1 4 \rho^2 \quad \text{in $B_{2\rho}$}\end{equation} for a polynomial $V_{{\bf m},p,q}$ with universally bounded coefficients. This implies that $U$ satisfies the conclusion of the lemma and we get a contradiction. Indeed, using that $U_t(x)$ is $C^{1,1}$ with respect to the variables $(x,t)$ in the annulus $B_{2\rho} \setminus B_{\rho/2}$, and denoting $\bar {\bf m}=-{\bf m}$,
$$\frac{\tilde{U}_{t+\eps p} - U_t}{\eps} = {\bf m} \cdot  \nabla U_t + q\p_{\theta} U_t + p \p_t(U_t) + O(\eps)$$
$$= {\bf m} \cdot \nabla U_0 + q\p_{\theta} U_0 + p \p_t(U_t)|_{t=0} + O(\eps) + O(\delta)$$
$$ = V_{{\bf m}, p,q} + O(\eps) + O(\delta)$$
with the constants for $O(\eps), O(\delta)$ depending on $\rho.$ In conclusion, as double-valued functions,

$$|U -\tilde{U}_{t+\eps p}| \leq \frac \eps 2 \rho^2 \quad \text{in $B_{2\rho} \setminus B_{\rho/2}$}$$
hence  
$$d(U,\tilde{U}_{t+\eps p}) \leq \frac \eps 2 \rho^2 \quad \text{in $B_{2\rho} \setminus B_{\rho/2}$}$$
and since the distance is subharmonic (Lemma \ref{duv}), the inequality extends to  $B_{2\rho}.$ 

\end{proof}

 We are now ready to prove our main theorem on the classification of global solutions.
 
 \medskip

\noindent
{\it Proof of Theorem \ref{mainintro}.}
Since $\p \{u^i>0\} \neq \emptyset$, by non-degeneracy
$$\max_{\p B_r} u^i \geq cr^2, \quad \text{for all large $r$'s.}$$ By Proposition \ref{not1}, any type II blow-down limit of $U$ must be equal to a rotation of the 2-homogeneous solution $U_0$. By Remark \ref{all}, given $\eps_1$ small to be made precise later, there exists $R=R(\eps_1)$ such that 
\begin{equation}\label{begin}d(U_r, U_0(e^{iq_r}x)) \leq \eps_1 \quad \text{in $B_1$, for all $r\geq R.$}\end{equation}
Given a large $r$ as above, we want to apply Proposition  \ref{IF} and iterate it. First we remark that, after relabeling $\rho$, the conclusion of the Theorem holds in the ball $B_\rho(\eps {\bf m})$. This means that the quadratic rescaling of $U$ from $B_{\rho r}(\eps {\bf m})$ to $B_1$ is approximated by a rotation of $U_{t_1}$, with error $\eps_1 \rho^\alpha$ in $B_1$, for some $|t_1|\leq C\rho^{-1} \eps_1.$ If $|t_1| \leq \delta$, we can apply the Improvement of Flatness again. We do so iteratively and obtain a sequence of nested balls $B_{r_k}(z_k)$, $r_k=r \rho^k$, in which the rescaling of $U$ from these  balls to the unit ball is approximated by a rotation of $U_{t_k}$, with error $\eps_1 \rho^{k\alpha}$ in $B_1$, for some $|t_k|\leq \delta.$ We stop the process when $|t_k|>\delta$. 

We prove by induction that $0 \in B_{r_{k}/2}(z_{k})$  and that the process cannot stop if $r_k \geq R.$ Indeed, if $r_k \geq R$ then, 
$U_{r_k}$ is $\eps_1$ close in $B_1$ to a rotation of $U_0$ (by \eqref{begin}) and $\eps_1$ close in $B_1(z_k/r_k)$ to $\tilde U_k$, that is a rotation of $U_{t_k}$ translated by $z_k/r_k \in B_{1/2}$, $|t_k| \leq \delta$. Thus, 
$$d(U_{0}(e^{iq_{r_1}}x), \tilde U_k)\leq 2\eps_1 \quad \text{in $B_{1/4}$.}$$
The inequality above implies that $$|t_k|, |z_k/r_k| \leq \eta(\eps_1), \quad \text{$\eta(\eps_1) \to 0$ as $\eps_1\to 0.$}$$
This means that $|t_{k+1}|\leq \delta$ and $|z_{k+1}/r_{k+1}| \leq 1/2$, provided $\eps_1$ is chosen small depending on $\delta$ and $\rho$. 

In conclusion, in any fixed ball $B_{R'}$ with $R' \geq R$, we can approximate the rescaling $U_{R'}$ in $B_1$ by a translation and rotation of $U_t$ with arbitrarily small error $\eps_1 \rho^{k\alpha}$ by starting the process above with $r=R'\rho^{-k}$. Hence, the desired conclusion follows.

\qed

\section{The perturbative case: $|F|$ small}\label{sec:PerturbativeCase}

In this section we prove the uniqueness of type I blow-ups and the regularity of the free boundary at such points.  Throughout the section $n=2, N=3$, and we denote by $\mathcal U_0$ the minimizer of $J_0$ with
$$|\mathcal U_0|=r^{3/2}|\cos \frac{3\theta}{2}|.$$

\begin{thm}[Uniqueness of type I blow-ups]\label{curves} Let $U$ minimize $J(\cdot, F)$ in $B_1$ with $F \geq 0$ and assume that for some $\eps\leq \eps_0$ small universal
\begin{equation}\label{ass0}d(U,\mathcal U_0) \leq \eps, \quad \|F\|_{\infty} \leq \eps^2 \quad \text{in $B_1$.}\end{equation} Then, after a translation of $U$ by $\eps  {\bf \tilde m}$, a rotation by $\eps \tilde q$, and multiplication by $(1+\eps \tilde p)$, we have
\begin{equation}\label{3/2}d(U,\mathcal U_{0}) \leq C\eps |x|^{3/2+\alpha} \quad \text{in $B_1$},\end{equation} with $|\tilde m|, |\tilde p|$ and $|\tilde q| \leq C.$ In particular, $\Gamma(U)$ consists of three $C^{1,\alpha}$ curves starting at $0$ where they meet at $120^\circ$ (triple junction).
\end{thm}

Theorem \ref{curves} is obtained by iterating the following improvement of flatness proposition.

\begin{prop}[Improvement of Flatness for type I cones]\label{zero} Let $U$ minimize $J(\cdot, F)$ in $B_1$ with $F \geq 0$. There exist $\eps_0, \rho>0$ universal such that if for some $\eps \leq \eps_0$

\begin{equation}\label{ass1}d(U,\mathcal U_0) \leq \eps, \quad \|F\|_{\infty} \leq \eps^2 \quad \text{in $B_1$,}\end{equation} 
then $$d(U,(1+\eps p)\tilde {\mathcal U}_{0}) \leq \frac \eps 2 \rho^{3/2} \quad \text{in $B_\rho(\eps\bar{\bf m})$},$$ with $\tilde {\mathcal U}_0$ a rotation by $\eps q$ and translation by $\eps \bar{\bf m}$ of $\mathcal U_{0},$ with $|\bar{\bf m}|, |p|$ and $|q| \leq C.$
\end{prop}

Notice that the proposition above can be iterated indefinitely, which gives the conclusion of Theorem \ref{curves}. This is because
$U_\rho:= \rho^{-3/2}U(\rho x)$ minimizes $J(\cdot, F_\rho)$ for $F_\rho (x)= \rho^{1/2}F(\rho x)$ and \eqref{ass1} is satisfied by $U_\rho$ in $B_1$ (modulo a translation by $O(\eps)$ and division by $1+ O(\eps)$) with $\eps$ replaced by $\eps/2$ since $$\|F_\rho\|_\infty \leq \rho^{1/2} \eps^2 \leq \frac 1 2 (\frac \eps 2)^2.$$

A consequence of Theorem \ref{curves} is that any global solution must have at most quadratic growth as claimed in Proposition \ref{not1}, in Section \ref{sec:Asympt-global}. If $F>0$ is constant then by non-degeneracy solutions must grow precisely at a quadratic rate and $\mathcal U_0$ cannot occur as a blow-down. In other words, the cone $\mathcal U_0$ can occur as a blow-up limit but not as a blow-down limit of a global solution.

\begin{proof}[Proof of Proposition \ref{not1} (see Section \ref{sec:Asympt-global})] 
Assume the conclusion does not hold. Then, by Lemma \ref{time}, the blow-down must be $\mathcal{U}_0$ (up to rotations). Hence, the rescalings $U_j(x) := \frac{U(r_j x)}{a_{r_j}}$ satisfy the hypothesis of Proposition \ref{zero}, which implies the conclusion of Theorem \ref{curves} holds (for a given $\epsilon$ and for all large $j$). 
In particular, by applying this Theorem for a fixed $U_j$ with $j$ large, we conclude that $U$ has a triple junction point. Without loss of generality we assume that this triple junction occurs at the origin. This means that all the $U_j$'s for all large $j$'s have a triple junction at the origin and therefore we do not need to translate $U_j$ in \eqref{3/2}.

Then, from \eqref{3/2}, we get
$$\max_{B_r}|U_j| \geq \frac 12 r^{3/2}, \quad r \leq 1.$$ 
Hence,
$$\frac{a_{r r_j}}{a_{r_j}} \geq \frac 1 2 r^{3/2}$$ and, by choosing $r=1/r_j$, we get
$$a_{r_j} \leq 2 a_1 r_j^{3/2}.$$ This contradicts the fact that $\lim_j \frac{a_{r_j}}{r_j^2} = +\infty.$
\end{proof}

We argue as in the previous section to associate to a minimizer $U$ of $J(\cdot, F)$ in $B_1$ that satisfies
$$d(U, \mathcal U_0) \leq \eps, \quad \|F\|_\infty \leq \eps^2,$$
a double-valued function still denoted by $U \in \mathcal D$ in an annulus $B_{1-\delta} \setminus B_\delta.$ In this process, property \eqref{UF} is guaranteed by the perturbation Lemma \ref{pert}. The function $U$ satisfies
$$|\Delta U| \leq \eps^2,$$
and we can let $\delta\to 0$ as $\eps \to 0.$

We immediately have the following.

\begin{lem}\label{improzero} Let $U$ be as above. Set 
$$V_\eps:= \frac{U-\mathcal U_0}{\eps},$$
For a sequence $\eps_k \to 0$, up to a subsequence, $V_{\eps_k}$ converges uniformly on compact subsets of $B_1 \setminus \{0\}$ to a double-valued function $V$ that satisfies 
$$\Delta V = 0 \quad \text{in $B_1 \setminus \{0\}.$}$$
\end{lem}

Lemma \ref{extend} applies in this context as well, hence $V$ is defined at zero and harmonic in $B_1$. The proof of Lemma \ref{expandV} contains the proof of the following expansion result. The difference here is that $U_0$ is replaced by $\mathcal U_0.$

\begin{lem}[Expansion in terms of homogeneous solutions $\mathcal U_0$]\label{expz} Let $V$ be a double-valued harmonic function in $B_1$ with $\|V\|_\infty \leq 1.$ Then 
$$V(x)= V_{{\bf m},p, q}(x)  + O(|x|^{5/2}) = 
{\bf m} \cdot \nabla \mathcal U_0 + p r^{3/2} \cos\frac{3\theta}{2} + q \p_\theta \mathcal U_0 + O(|x|^{5/2}),$$
where $|{\bf m}|, |p|$ and $|q|$ are universally bounded.
\end{lem}

With these tools, the proof of the improvement of flatness Proposition \ref{zero} follows with the same arguments as in Proposition \ref{IF} and we omit it.

Next we use that $\frac 1 t U_t \to \mathcal U_0$ in $B_1$ as $t \to \infty$ and obtain a version of Proposition \ref{IF} for large values of $t$.

\begin{lem}\label{IF2} Let $U$ be a minimizer of $J(\cdot,F)$, $F=(1,1,1)$ and assume that 
$$ d(U,U_t) \le \eps t \quad \mbox{in} \quad B_1,$$
for some $t \ge \delta^{-1}$, $\eps \le \eps_0$. Then, for $t'=(1+ p \eps)t$,
$$d(U,\tilde U_{t'}) \le \frac \eps 2 t' \rho^{3/2} \quad \mbox{in} \quad B_\rho(\eps \bar {\bf m}),$$
with $\tilde U_{t'}$ a rotation by $\eps q$ and translation by $\eps \bar {\bf m}$ of $U_{t'},$ with $|\bar{\bf m}|, |p|$ and $|q| \leq C.$
\end{lem} 

The conclusion of the lemma can be iterated indefinitely since the rescaling $U_\rho(x)=U(\rho x)/\rho^2$ satisfies (modulo a translation and rotation of order $O(\eps)$) the hypotheses with $t$ and $\eps$ replaced by $t' \rho^{-1/2}$ and $\eps/2$. 

\begin{proof}
We divide by $t$, and then the rescaled error function 
$$V_{\eps,t}=\frac{\frac 1 t U- \frac 1 t U_t}{\eps},$$
by Lemma \ref{improgeneral} converges (on subsequences) as $t \to \infty$, $\eps \to 0$ uniformly on compact sets of $B_1 \setminus \{0\}$ to a double-valued harmonic function $V \in \mathcal D$. By Lemma \ref{expz}, we have
$$|V-V_{{\bf m},p,q}| \le \frac 1 8 \rho^{3/2} \quad \mbox{in} \quad B_{2\rho}.$$
The conclusion follows by using that (see \eqref{36})
$$\frac 1 t U_t \to \mathcal U_0 \quad \mbox{in} \quad C^1(B_{2 \rho} \setminus B_{\rho/2}), \quad \quad t\frac {d}{dt} (\frac 1 t U_t) \to 0 \quad \mbox{uniformly in $B_{2 \rho} \setminus B_{\rho/2}$},$$ 
as $t \to \infty$, hence
\begin{align*}
\frac 1 t \tilde U_{t'}& =(1+p \eps)\frac {1} {t'} \tilde U_{t'}\\
&=\frac 1 t \tilde U_t + p \eps \frac {1} {t'} \tilde U_{t'} + o(1)\eps\\
&= \frac 1 t U_t + \eps q \p_\theta \mathcal U_0 + \eps {\bf m} \cdot \nabla \mathcal U_0 + p \eps \mathcal U_0 + o(1)\eps\\
&= \frac 1 t U_t+ \eps V_{{\bf m},p,q} + o(1)\eps,
\end{align*}
with $o(1) \to 0$ as $t \to \infty$, which gives the desired conclusion in the annular region $B_{2\rho} \setminus B_{\rho}$. This can be extended to the whole region because of the subharmonicity of $d(U, \tilde{U}_{t'}).$

\end{proof}

Combining Proposition \ref{IF} with Lemma \ref{IF2} we find that an $\eps$ perturbation of $U_0$ can be well approximated by one solution $U_t$ at all scales.

\begin{prop}[Perturbations of type II cones]\label{AS} Let $U$ be a minimizer of $J(\cdot,F)$, $F=(1,1,1)$ and assume that we have
$$ d(U,U_0) \le \eps \quad \mbox{in} \quad B_1,$$
for some $\eps \le \eps_1$ small. Then, after a translation and rotation of $U$ by a factor of $O(\eps)$, we have 
$$d(U,U_t) \le C \eps r^{\alpha}\|U_{t^+}\|_{L^\infty(B_r)},$$
for some $|t| \le C \eps$, and all $r \le 1 $.

In particular, if the origin is a triple junction point with a type II blow-up, then the conclusion holds with $t=0$. 
\end{prop}

 The conclusion can be restated that the solution $U$ is obtained from a $U_t$ by a $C^{1,\alpha}$ diffeomorphism $\Phi$:
 $$U( \Phi(x))=U_t(x), \quad \quad [D \Phi]_{C^\alpha} \le C \eps, \quad D \Phi(0)=I.$$

\begin{proof} We can iterate Proposition \ref{IF} and obtain a sequence of nested balls $B_{r_k}(z_k)$ and $t_k$ such that
$$d(U, \tilde U_{t_k}) \le \eps r_k^{2+\alpha} \quad \mbox{in} \quad B_{r_k}(z_k),$$
$$r_k=\rho^k, \quad |t_{k+1}-t_k| \le C \eps r_k^{\frac 12 + \alpha}, \quad |z_{k+1}-z_k| \le C \eps r_k^{1+\alpha}, \quad z_0=0, \quad t_0=0,$$
and $\tilde U_{t_k}$ denotes a rotation of $U_{t_k}$ by angle $q_k$, and then translated at $z_k$, with $|q_{k+1}-q_k| \le C \eps r_k^\alpha$.

The iteration of Proposition \ref{IF} stops the first time when 
\begin{equation}\label{tkrk}|t_k|r_k^{-1/2} > \delta.
\end{equation}If the iteration continues indefinitely, then $t_k \to 0$, and we obtain the conclusion with $t=0$.
Otherwise, let $k_0$ be the first $k$ when \eqref{tkrk} holds. Notice that for $k=k_0, k_0+1,\ldots,k_0+m_0$ for some $m_0$ large depending on $\delta$ we have
$$d(U, \tilde U_{t_{k_0}}) \le C(\delta) \eps r_k^{2+\alpha} \quad \mbox{in} \quad B_{r_k}(z_{k_0}),$$
and $$|t_{k_0}| r_{k}^{-1/2} > \delta^{-1} \quad \mbox{for the last value $k=k_1:=k_0+m_0$}.$$
If $t_{k_0}<0$, then both $U$ and $\tilde U_{t_{k_0}}$ vanish in the last ball of radius $r_{k_1}$ and the conclusion of the proposition is satisfied with $t=t_{k_0}$.

If $t_{k_0}>0$, then in the last ball of radius $r_{k_1}$, we can start iterating Lemma \ref{IF2} indefinitely, and find that for all $k \ge k_1$
$$d(U, \tilde U_{t_k}) \le C' \eps \, t_k \, r_k^{ \frac 3 2 +\alpha} \quad \mbox{in} \quad B_{r_k}(z_k),$$
$$|t_{k+1}/t_k -1| \le C'' \eps r_k^{\alpha}, \quad |z_{k+1}-z_k| \le C'' \eps r_k^{1+\alpha}, \quad |q_{k+1}-q_k| \le C'' \eps r_k^\alpha,$$
which gives the desired conclusion for some $t=t_{k_0} + O(\eps r_{k_0}^{\frac 12 +\alpha})$.

\end{proof}

Combining the proposition above with Theorem \ref{curves}, we conclude by compactness that the following holds for $F=(1,1,1)$. The general case for $F>0$ is postponed till Section \ref{sec:non-sym-F}.

\begin{thm}\label{thm:best} Assume that $n=2$, $N\ge 3$ and that $U$ is a minimizer of $J(\cdot,F)$ with $F>0$ constant. There exist $r_0$ small and $C$ large depending on $M=\|U\|_{L^\infty(B_1)}$, $F$, and $N$ such that if at least 3 different components $\{u^i>0\}$ intersect $B_{r_0}$ then after a rotation and translation $U$ is well-approximated by a global solution $U_t^F$ with $|t| \le C$, in the following sense
$$d(U,U_t^F) \le C r^\alpha \|U^F_{t^+}\|_{L^\infty (B_r)}, \quad \quad \forall r \le 1.$$ 
\end{thm}

Finally, we conclude the section with the proof of Theorem \ref{localclassintro} for the case $F=(1,1,1)$.
\medskip

{\it Proof of Theorem \ref{localclassintro}.} Under the assumptions of the theorem, up to a subsequence and a rotation, either $$U_r(x):= a_r^{-1} U(rx) \to \mathcal U_0$$
or 
$$U_r(x):= r^{-2} U(rx) \to U_0.$$

In the first case, we choose $r_0$ small enough so that 
\begin{equation*}d(U_{r_0}, \mathcal U_0) \leq \eps_0 \quad \text{in $B_1$}\end{equation*} and the hypotheses of Theorem \ref{curves} are satisfied. This gives the desired conclusion.

In the second case, we can iterate Proposition \ref{IF} starting with $r_0$ small and $t_0=0$ so that 
\begin{equation*}d(U_{r_0}, U_0) \leq \eps_1 \quad \text{in $B_1$}.\end{equation*} Thus,
$$d(U_{r_k}, \tilde U_{t_k}) \leq \eps_1 2^{-k} \quad \text{in $B_1$}, \quad r_k=r_0 \rho^k,$$
as long as $$|t_k| \leq \delta.$$ If we can iterate indefinitely, we reach the desired conclusion again. Otherwise, if the iteration stops for $t_k<-\delta$ then we contradict $0$ as a triple junction provided that we choose $\eps_1$ small enough depending on $\delta$. If $t_k >\delta$, then after a further rescaling by a large factor we end up in a case in which Theorem \ref{curves} applies and reach our conclusion.
\qed

\section{The non-symmetric case of constant $F > 0$}\label{sec:non-sym-F}

In this section we prove our main results for $F>0$ constant. The proofs are identical to the ones of the symmetric case, except that now the construction of the family 
 $U_t^F$  is more delicate. Therefore most of this section is devoted to the construction of $U_{\pm1}^F$ as perturbations of $U_0^F$ with specific asymptotics at infinity.
 
 Precisely, let 
 $V$ satisfy the linearized equation

\begin{align}\label{FBgeneral}
& \Delta V=0 & \text{in $\{U_0^F\neq 0\}$} , \\
\label{transmissiongen} & V_{\nu^+} + V_{\nu^-} =\frac{a^F}{|\nabla U^F_0|} V & \text{on $\{U_0^F=0\},$} 
\end{align} with $a^F>0$ the sum of the values of $F$ on the two adjacent sides of a given ray in $\{U_0^F=0\}.$

\begin{lem}\label{ei}Assume $F$ has at least two distinct components. Let $\lambda^F_k$ denote the sequence of homogeneities of the operator associated with \eqref{FBgeneral}-\eqref{transmissiongen} on $\p B_1$.Then $$\lambda^F_1=\lambda^F_2=1 \quad \text{with eigenfunctions generated by $\p_{x_1} U_0^F, \p_{x_2} U_0^F$,}$$ $$\lambda^F_3 \in (3/2, 2) \quad \text{with eigenfunction $g_3^F$,}$$
$$\lambda^F_4=2 \quad \text{with eigenfunction $\p_\theta U_0^F$}$$
and $\lambda^F_5>5/2$.
\end{lem}

\begin{proof} By the min-max characterization as in \eqref{quotient}, we can see that $$\lambda_m^F \geq \lambda_m^0$$ and in particular (see \eqref{zeroeigen})$$\lambda^F_{2k+1} \geq\frac{2k+1}{2}.$$ The equality only holds if a linear combination of $\cos \frac{2k+1}{2} \theta$ and $\sin \frac{2k+1}{2} \theta$ vanishes where $U^F_0=0$ on the unit circle, which means that after rotation, the rays of $U_0^F$ belong to $\theta\in \{2l\pi/(2k+1), \;l=0,1,..,2k\}$. This cannot happen if $k=1$ and $F$ has two distinct components or $k=2$ (see \eqref{uofintro}). Thus, $ \lambda_3^F>\frac32$, $\lambda_5^F>\frac52$.

 We now claim that $2$ is a simple eigenvalue for our problem. Assume without loss of generality that  $U_0^F = 0$ consists of the rays $\theta \in \{0,\theta_1,\theta_2\}$ with $\theta_1 \in (\pi/2, \pi), \theta_2 \in (\pi, 3\pi/2)$ and $\theta_2-\theta_1 > \pi/2.$
Then, if $2$ is not simple, we can find an eigenfunction $g$ that vanishes on one ray of $U_0^F = 0$, say $\theta = 0$. Recall that,
\begin{equation*}\text{$g$ is periodic of period $4\pi$ and $g(\theta+2\pi) = - g(\theta)$.}\end{equation*} Therefore $g$ coincides with $\sin 2\theta$ in the interval $[\theta_2 - 2\pi, \theta_1],$ and 
$$g(\theta_1)= \sin 2\theta_1 <0, \quad g(\theta_2)= -g(\theta_2 -2\pi)=-\sin(2\theta_2 - 4\pi)<0.$$ On the interval $[\theta_1, \theta_2]$, $$g(\theta)= a \cos(2(\theta-\tau))$$ for some $a,\tau.$ Given the behavior of $g$ at $\theta_1, \theta_2$ and the fact that $\theta_2-\theta_1>\pi/2$, we deduce that $a>0$ and $\tau \in [\theta_1,\theta_2].$ Without loss of generality we can assume that $\tau \leq \pi$ and $\theta_1$ is the point of intersection of the graphs of $\sin 2\theta$ and $a \cos(2(\theta -\pi))$ and therefore belongs to the interval $(\pi/2, 3\pi/2).$ Thus, we cannot satisfy the free boundary condition \eqref{g'} since 
$$g'(\theta_1^-)<0, \quad g'(\theta_1^+)>0$$
and hence their difference is positive while $g(\theta_1)<0.$ We conclude that $2$ is a simple eigenvalue. Since $\lambda^F_m$ varies continuously with $F,$ we deduce that $\lambda_3^F$ and $\lambda_4^F$ are both simple, hence $\lambda_3^F<2.$

\end{proof}

The eigenfunction $g^F_3$ is chosen so that its $L^2$ norm on the unit circle is $1$. Since the third eigenvalue is always simple, its eigenfunction $g_3^F$ varies continuously with $F$. For our discussion, we choose its sign so that $g_3^F$ is obtained as a continuous deformation from $\cos \frac 3 2 \theta$, which corresponds to $F=(1,1,1)$ (see Section \ref{sec:proof-symmetric}).

In view of Lemma \ref{ei} we obtain the same asymptotic expansion as in the case $F=(1,1,1),$ by using the corresponding polynomial solutions
$$V_{{\bf m},p,q}^F= {\bf m}\cdot \nabla U_0^F + p|x|^{\lambda_3}g_3^F(\theta) + q\p_\theta U_0^F.$$

\begin{lem} \label{L8.2} Let $V$ be a solution to \eqref{FBgeneral}--\eqref{transmissiongen} in $B_1$. Then, 
$$V=V^F_{{\bf m}, p, q} + O(|x|^{5/2}),$$ for some ${\bf m}, p, q.$
\end{lem}

Next, we construct global minimizers to $J(\cdot, F)$ which differ from $U_0^F$ by $r^{\lambda_3} g_3^F(\theta)$ plus lower order terms.

\begin{prop}\label{global} There exists a global minimizer $U^F_1$ to $J(\cdot, F)$ such that outside a ball $B_C$, viewed as a double-valued function, has the expansion
$$U^F_1=U^F_0+ |x|^{\lambda_3}g_3^F(\theta) + O(|x|^{\gamma}),$$ with $\gamma \in (2\lambda_3-2, \lambda_3) \subset (1,2).$ Similarly, there exists a global minimizer $U_{-1}^F$ with expansion 
$$U^F_{-1}=U^F_0- |x|^{\lambda_3}g_3^F(\theta) + O(|x|^{\gamma}),$$ outside $B_C.$ 
\end{prop}
The constant in the error term $O(|x|^\gamma)$ and the constant $C$ in the expansions above depend only on $F, \gamma$, once we fix $\gamma$ in the given interval. 

The proof of this proposition is provided in the next Subsection.

\subsection{Approximate Global Solutions}\label{subsec:Approx-global}
First we give a version of Proposition \ref{global} in which we construct an approximate global solution with error that improves as $|x|\to \infty$ (see \eqref{improvederror}). Recall that $\lambda_3 \in (\frac 3 2, 2).$

\begin{lem}\label{G} There exists a global $C^{1,1}$ (double-valued) function $G$ that has the asymptotic expansion
$$G=U^F_0+ |x|^{\lambda_3}g^F_3(\theta) + O(|x|^{2\lambda_3-2}), \quad |x| \geq 1$$ such that it solves the original problem with an error $O(|x|^{2\lambda_3-4})$. That is, if the $g^i$'s denote the restriction of $|G|$ (which is single-valued) to the three connected components of $|G|>0$, then 
\begin{equation}\label{improvederror}
    \Delta g^i=f_i \chi_{\{g^i>0\}}+ O(|x|^{2\lambda_3-4}).\end{equation}
\end{lem}
\begin{proof} The existence of such a function $G$ is a consequence of the fact that $|x|^{\lambda_3}g_3^F(\theta)$ solves the linearized equation. We provide the details below.

Assume as in the previous proof that $\{U_0^F=0\}$ consists of $\theta \in \{0,\theta_1, \theta_2\}.$ Let
$$u=U^F_0, \quad v= |x|^{\lambda_3}g_3^F(\theta).$$ When $\theta$ stays at distance $c_0$ away from $0, \theta_1, \theta_2$ we define 
$$G:=u+ v$$ 
which  satisfies  the desired equation with no error. In a neighborhood of a ray, say $\theta=0$, we define $G$ as a domain deformation of $u$:
$$G\left(y\right)=u(x), \quad y:=x- \frac{v}{\p_{x_2}u}e_2.$$
We glue the two constructions using a standard partition of unity function in the $\theta$ variable, and at the end we multiply the function with a cut-off function in the radial variable to make $G$ identically zero in a large ball around the origin.

\

Let's check that  such a function satisfies the desired equation in a neighborhood of the rays at a point at distance $r>1$ from the origin. 

 After a quadratic rescaling $G(rx)/r^2$, the rescaled function $G_\eps$ in $B_2 \setminus B_1$ away from the ray $\theta=0$ takes the form $$G_\eps= u+ \eps v, \quad \eps=r^{\lambda_3-2},$$ and in a neighborhood of the ray $$G_\eps\left(y\right)=u(x), \quad y:=x-\eps \frac{v}{\p_{x_2}u}e_2.$$

We check that in the second case
\begin{equation}\label{deltageps}\Delta G_\eps(y)=\Delta u(x) + O(\eps^2),\end{equation} away from $\theta=0.$
This follows from the chain rule and the fact that $\Delta u$ is constant and $v$ is harmonic on either side of the ray $\theta=0$. Notice that the change of variable $x \to y$ is of class $C^{1,1}$ across the ray since the normal derivatives of $\frac{v}{\p_{x_2} u}$ are equal due to the free boundary condition. 

We provide some details to justify \eqref{deltageps}. Given a domain deformation
$$x \to y= x+ \eps \Phi(x),$$ 
since 
$$D_y x = I - \eps D\Phi + O(\eps^2), \quad D^2_y x = -\eps D^2 \Phi + O(\eps^2),$$
we compute using the chain rule the Laplacian of $u$ with respect to $y$:
$$\Delta_y u = \Delta_x u -\eps(2 u_{ij}\Phi_i^j+u_j\Phi^j_{ii}) + O(\eps^2)=$$
$$= \Delta_xu -\eps (\nabla u \cdot \Phi)_{ii} + \eps (\Delta u)_j\Phi^j +O(\eps^2).$$ This gives the equation \eqref{deltageps}.

Moreover, using that $$x= y-\eps \Phi$$ and expanding $u$ around the point $y$, we get the expansion
$$G_\eps(y)= u(y) + \eps v(y) + O(\eps^2).$$
\end{proof}

\begin{defn}[Family of approximate solutions $G_t$.]\label{def:approx-sol}
We denote by $G_t$ the following rescaling of $G$, \begin{equation}\label{Gt}G_t(x):=t^{2\beta}G(t^{-\beta}x),\quad \beta= \frac{1}{2-\lambda_3}\end{equation} with $G$ the function from the lemma above. Notice that  
$$G_t=U^F_0+ t r^{\lambda_3}g_3^F(\theta) + O(t^2r^{2\lambda_3-2}), \quad r \geq t^\beta.$$
Moreover $|G_t|$ is single-valued and solves the Euler-Lagrange equation with an error of order $O(t^2r^{2\lambda_3-4}).$
\end{defn}

Next, we prove  an improvement-of-flatness-type result, using the family $G_t$. Precisely, we show that if a minimizer $U$ is well approximated by $G_t$ in a ball of radius $r$, then the same happens in a ball of radius $\rho r$ for some $\rho$ universal.

\begin{lem}\label{impapp} Let $U$ be a minimizer of $J(\cdot, F)$ in $B_1$, and fix $\alpha \in (0,1)$. Assume that for $\eps \leq \eps_0$, $\eps_0$ small depending on $\alpha$ and $F$, 
$$d(U, |G_\eps|)\leq \eps^{1+\alpha} \quad \text{on $\p B_1$.}$$ 
Then after a translation of $U$ by a vector of length at most $C\eps^{1+\alpha}$ we have
$$d(U, |G_{\eps'}|)\leq \rho^2(\eps'\rho^{\lambda_3-2})^{1+\alpha} \quad \text{in $B_\rho$,}$$ with $|\eps'-\eps|\leq C\eps^{1+\alpha}$ and $\rho$ universal.
\end{lem}

\begin{rem}\label{iterate} The conclusion in the lemma above can be iterated (after a translation of $U$) in $B_\rho$ since the quadratic rescaling $ \rho^{-2}U(\rho x)$ is approximated in $B_1$ with error $\eps_1=\eps'\rho^{\lambda_3 -2}.$ Equivalently, we can write the statement of the lemma in an unrescaled setting in the following way. If
$$d(U, |G_t|) \leq t^{1+\alpha}r^\gamma, \quad \text{in $B_r$} \quad
tr^{\lambda_3 - 2} \leq \eps_0,$$ with the relationship between $\alpha$ and $\gamma$ given by
$$\gamma=2+(\lambda_3 - 2)(1+\alpha),  \quad \gamma \in (2\lambda_3-2, \lambda_3)$$
then, 
$$d(U, |G_{t'}|) \leq t'^{1+\alpha}(\rho r)^\gamma \quad \text{in $B_{r\rho}(z)$}$$ with 
$$|t' - t| \leq C t^{1+\alpha}r^{(\lambda_3-2)\alpha}$$
and 
$$|z|\leq C t^{1+\alpha} r^{\gamma-1}.$$

\end{rem}

\begin{proof} We apply Remark \ref{Appppp} for $U$ and $|G_\eps|$. Thus,
$$\Delta d(U,|G_\eps|) \geq - C(\eps^2|x|^{2\lambda_3-4})\geq C\eps^2\Delta(1-|x|^{2\lambda_3-2}).$$ By the maximum principle, 
$$d(U, |G_\eps|) \leq \eps^{1+\alpha}+ C\eps^2 \leq 2\eps^{1+\alpha}, \quad \text{in $B_1$}.$$ This means that for each $\delta>0$ small, $\Gamma(U) \cap (B_1\setminus B_\delta)$ is sufficiently close to $\Gamma(U_0^F)$ provided that $\eps$ is small depending on $\delta.$ As before, in this annulus we identify $U$ with a double-valued function and look at the error
$$V_\eps:= \frac{U-G_\eps}{\eps^{1+\alpha}}, \quad \|V_\eps\|_{L^\infty(B_1\setminus B_\delta)} \leq 2.$$ Then, for a sequence of $\eps_k \to 0$, $V_{\eps_k}$ converges by Lemma \ref{improgeneral} (see also Remark \ref{remlast}) to a solution of the linearized problem associated to $U_0^F$, in $B_1 \setminus B_\delta$. Here we have used that $G_\eps$ solves our problem with an error $O(\eps^2)$ and that $\alpha<1.$ Moreover, as $\delta \to 0$, this limit solution converges to the solution of the linearized equation in the full ball $B_1$ (see Lemma \ref{extend}). 
In conclusion, if $\eps$ is sufficiently small, $$V_\eps = {\bf m} \cdot \nabla U_0^F + p r^{\lambda_3}g_3^F(\theta) +O(\rho^2), \quad \text{in $B_{2\rho} \setminus B_\delta$}.$$ Thus, from the definition of $V_\eps$, since for $\eps \to 0$, $\nabla G_\eps \to \nabla U_0^F$
$$U=G_{\eps'}(x+\eps^{1+\alpha}{\bf m}) + O(\eps^{1+\alpha}\rho^2), \quad \eps'=\eps+p\eps^{1+\alpha}, \quad \text{in $B_{2\rho} \setminus B_\delta$}.$$ Thus, the translation $U(x-\eps^{1+\alpha}m)$ is approximated by $G_{\eps'}$ in $B_\rho \setminus B_\delta$ with an error $O(\eps^{1+\alpha}\rho^2)$
and this can be extended from $\p B_\rho$ to the full $B_\rho$ with the same maximum principle argument as before. We get the desired estimate by choosing $\rho$ sufficiently small, since $\lambda_3<2.$
\end{proof}

Finally, we are ready to prove the existence of global minimizers in Proposition \ref{global}.

\begin{proof}[Proof of Proposition \ref{global}.] In a large ball $B_R$, let $U$ minimize $J(\cdot, F)$ with boundary data the function $G$ from Lemma \ref{G}. By applying Lemma \ref{impapp} (see Remark \ref{iterate}) we find a sequence of nested balls $$B_{r_k}(z_k), \quad r_k=R\rho^k$$ and a sequence of $t_k \in (\frac 1 2, 2)$ such that
$$d(U, |G_{t_k}|(x-z_k)) \leq t_k^{1+\alpha}r_k^\gamma \leq Cr_k^\gamma \quad \text{in $B_{r_k}(z_k)$}$$
$$|t_{k+1} - t_k| \leq C t_k^{1+\alpha}r_k^{(\lambda_3-2)\alpha}$$
and 
$$|z_{k+1}-z_k|\leq C t_k^{1+\alpha} r_k^{\gamma-1},$$ as long as
$$r_k \geq C(\eps_0) \quad \text{large}.$$
Let $\bar k$ be the index for which that inequality above is satisfied for the last time, and call $\bar z:=z_{\bar k}.$ Then, from the inequalities above and the geometric behavior of the right hand sides it follows that 
$$|z_k-\bar z|\leq C r_k^{\gamma -1}, \quad \forall k\leq \bar k,$$
and
$$|t_{k}-1|\leq C r_{k}^{(\lambda_3-2)\alpha}, \quad \forall k\leq \bar k.$$ Thus,
$$d(U, |G_1|(x-\bar z)) \leq C r_k^\gamma \quad \text{in $B_{r_k}(\bar z),$}$$ as long as $k \leq \bar k.$ In conclusion 
$U_R(x):=U(x+\bar z)$ is a minimizer which is defined in $B_{R/2}$, satisfies 
$$d(U_R, |G|) \leq C r^\gamma \quad \text{in } B_{r}, R/2 \geq r\geq C(\eps_0).$$ Now the conclusion follows by letting $R\to \infty$ and extracting a convergent subsequence of the $U_R$'s.
\end{proof}

%%%%%%%%%%%
%%%%%%%%%%%

\subsection{Proof of the main theorems for general $F$}\label{subsec:proof-main-thm}

First we define the family of solutions $U_t^F$ by the rescalings
\begin{align}\label{UFt}
U^F_t(x)&:=t^{2\beta} U_{1}^F(t^{-\beta}x), \quad t>0 \\
\nonumber U^F_t(x)&:=|t|^{2\beta} U_{-1}^F(|t|^{-\beta}x), \quad t<0.
\end{align}
Notice that $U^F_t$ has the following expansion,
\begin{equation}\label{asy}U_t^F=U^F_0+ t r^{\lambda_3}g_3^F(\theta) + O(t^{\beta(2-\gamma)}r^{\gamma}), \quad r \geq t^\beta,\end{equation} where $\beta(2-\gamma)>1=\beta(2-\lambda_3).$

We provide a version of Proposition \ref{IF} for general $F$. Precisely, 
\begin{prop}[Improvement of Flatness for general $F$]\label{IFgeneral}Let $U$ minimize $J(\cdot, F)$ in $B_1$, $F > 0$ constant. There exists $\eps_0, \delta, \rho, C$ depending on $F$, such that if for $\eps \leq \eps_0$
$$d(U,U^F_t) \leq \eps, \quad |t| \leq \delta, \quad \text{in $B_1$}$$
then $$d(U,\tilde U^F_{t+\eps p}) \leq \frac \eps 2 \rho^2 \quad \text{in $B_\rho(\eps \bar{\bf m})$},$$ with $\tilde U^F_{t+\eps p}$ a rotation by $\eps q$ and translation by $\eps \bar{\bf m}$ of $U^F_{t+\eps p},$ with $|\bar{\bf m}|, |p|$ and $|q| \leq C.$
\end{prop}
\begin{proof}The proof of this proposition is identical to the proof of Proposition \ref{IF} once we show
that 
\begin{equation}\label{Uc11}\text{$U^F_t(x)$ as a function of the $(x,t)$ variables is $C^1$ in $B_1\setminus B_{\rho/2} \times [-\delta, \delta].$}\end{equation} 
This is clear when $t\neq 0$, according to formula \eqref{UFt} and the fact that $U_1^F$ is $C^{1,1}$ outside a fixed neighborhood of the origin. Thus, we only need to check the continuity of the derivatives as $t\to 0.$

After we establish \eqref{Uc11}, from the asymptotic expansion of $U_t^F$ we have
\begin{equation}\label{gradexp}\p_t U_t^F |_{t=0}= r^{\lambda_3}g_3^F(\theta), \quad 
\nabla U_t^F(x)|_{t=0}= \nabla U_0^F(x).\end{equation}Then 
we can conclude that in the annulus $B_{2\rho} \setminus B_{\rho/2}$, $\bar {\bf m}=-{\bf m},$
$$\frac{\tilde{U}^F_{t+\eps p} - U^F_t}{\eps} 
= {\bf m} \cdot \nabla U^F_0 + q\p_{\theta} U^F_0 + p \p_t(U^F_t)|_{t=0} + o(1)$$
$$ = V_{{\bf m}, p,q} + o(1),$$
with
$$o(1) \to 0, \quad \text{as $\eps,\delta \to 0,$}$$ and proceed as before.

Next, we prove \eqref{Uc11}. Let us compute 
$$\p_t U^F_t= t^{-1} \beta (2 U_t^F - x\cdot \nabla U_t^F).$$
For notational simplicity, set 
\begin{equation}\label{uu0v}u:=U_t^F, \quad u_0:= U_0^F, \quad |x|^{\lambda_3}g_3^F(\theta)\end{equation} and, from the asymptotic expansion \eqref{asy}, \begin{equation}\label{exp1}u-u_0=tv+O(t^{1+\alpha}) \quad \text{in $B_1\setminus B_{\rho/2}$}.\end{equation} Away from a neighborhood of the rays $\{u_0=0\}$, the error $O(t^{1+\alpha})$ is harmonic hence also its derivatives are bounded by $O(t^{1+\alpha})$, and therefore
\begin{equation}\label{doubleconv}\nabla u \to \nabla u_0, \quad t^{-1}\beta(2u-x\cdot\nabla u) \to v\quad \text{as $t\to 0$,}\end{equation} which establishes our claim in this region.

In a neighborhood of the rays $\{u_0=0\}$, say around $\theta=0$, we use that the two solutions $u$ and $u_0$ differ by $O(t)$, and therefore one is obtained from the other by a $C^{1,\alpha}$ domain deformation of norm bounded by $Ct$ (see the proof of Lemma \ref{impro} in Section \ref{sec:Proof-Lemma-impro}, equation  \eqref{dd},)
\begin{equation}\label{dd1}u(x)=u_0(x+t\phi_t(x)e_2), \quad \|\phi_t\|_{C^{1,\alpha}}\leq C,\end{equation}
with $C$ independent on $t$. 

Combining equation \eqref{dd1} with \eqref{uu0v} we obtain, \begin{equation}\label{phitt}\phi_t = \frac{v}{\p_2 u_0} + O(t^\alpha),\end{equation} and since $\frac{v}{\p_2 u_0}$ is a $C^{1,1}$ function, we deduce that the derivative of the error $O(t^\alpha)$ is of order $O(t^\sigma)$ for some $\sigma$ small.

Clearly \eqref{dd1} gives $\nabla u \to \nabla u_0$ as $t\to 0$ also in this region. To show that also the second part of \eqref{doubleconv} holds in this region, we use that $u_0$ is homogeneous of degree 2 and obtain
$$2u-x \cdot \nabla u= \frac{d}{ds}|_{s=1}s^2u(x/s)= \frac{d}{ds}|_{s=1} u_0(x+ts\phi_t(x/s)e_2)$$
$$= \p_2u_0(x+t\phi_t) \cdot t \cdot \frac{d}{ds}|_{s=1}s \phi_t(x/s)$$
$$=  \p_2 u_0 \cdot t \cdot (2-\lambda_3) \frac{v}{\p_2 u_0} + O(t^{1+\sigma}),$$
where in the last equality we used that the function on the right hand side of \eqref{phitt} is homogeneous of degree $\lambda_3-1$. In conclusion
$$2u-x \cdot \nabla u=t \beta^{-1} v + O(t^{1+\sigma}),$$
which proves the second part of \eqref{doubleconv} as desired. 

\end{proof}

Now the proof of Theorem \ref{mainintro} is identical to the case $F=(1,1,1)$ since it is based on Proposition \ref{IFgeneral}, and we omit it.

The functions $U_{\pm1}^F$ that we constructed in Proposition \ref{global} were obtained by a limiting procedure by solving a sequence of problems in larger and larger domains, and are characterized by their asymptotic expansion at infinity. On the other hand their translations have the same expansion and it is not clear how to distinguish between two such nearby translations in our construction. In other words, the functions $U_{\pm 1}^F$ are defined uniquely only up to translations. In the next lemma we show that $U_1^F$ has a triple junction free boundary point while $U_{-1}^F$ does not. This means that we may define $U_1^F$ uniquely by choosing the triple junction point to be at the origin.  

\begin{lem} $\Gamma(U_1^F)$ has a unique triple junction and in a neighborhood of the junction it consists of three $C^1$ curves that meet at $120^\circ$. $\Gamma(U_{-1}^F)$ does not have a triple junction point. \end{lem}
\begin{proof}First we remark that the family $U_1^F$ is continuous in $F$ up to translations. Precisely, let $F_n \to F$, with $F_n, F$ constant and $F>0$. Then, up to subsequences, by Proposition \ref{IFgeneral}, $U_1^{F_n} \to W$ with $W$ a global solution for $J(\cdot, F)$ with expansion
$$W=U_0^F + |x|^{\lambda_3}g_3(\theta) + O(|x|^\gamma).$$
Thus, by the classification of global solutions, 
$$W=U_{1}^F(\cdot+{\bf m}).$$

Next, we show that the set of $F$'s in the first octant for which $\Gamma(U_1^F)$ has a triple junction is open and closed. Closedness is just a consequence of the uniform non-degeneracy (see Lemma \ref{ND}). To prove that this set is open, we notice that a triple junction point needs to have a type I blow-up profile. Indeed, if it has a type II blow-up profile, then it must coincide with a rotation of $U_0^F$. Then, since $U_1^F$ is asymptotic to the same cone $U_0^F$ at infinity, by Weiss monotonicity formula it must coincide with it. This is a contradiction hence at a triple junction $U_1^F$ has a type I blow-up profile. Since at infinity $U_1^F$ is asymptotic to $U_0^F$ we conclude that the sets $\{(u_1^i)^F>0\}$ are connected and therefore the type I blow-up profile coincides up to rotation with $\mathcal U_0=r^{3/2}\cos\frac 3 2 \theta.$ Now, by Theorem \ref{curves} and the continuity of $U_1^F$ up to translations, we conclude that our set is open and that near the triple junction (see Theorem \ref{curves}) the free boundary consists of $C^1$ curves meeting at $120^\circ$. 

The triple junction point is unique because of topological considerations. Indeed, since $\{(u^i_1)^F>0\}$ are connected, there are curves that lie inside these sets and connect the triple junction point to infinity. In particular this shows that no other triple junction point can exist.

The argument above applies also for $U_{-1}^F$, and since $F=(1,1,1)$
is not inside this set, this forces the set of $F$'s for which $U_{-1}^F$ has a triple junction to be empty. 
\end{proof}

In view of the lemma above, we can define $U_1^F$ uniquely, so that the triple junction occurs at the origin. We can also choose $U_{-1}^F$ so that the origin is interior to $\{U_{-1}^F=0\}$.

Since $U_1^F$ has type I blow-up at the origin it follows that the proofs of Lemma \ref{IF2} and Proposition \ref{AS} apply also in the general case $F>0$. We only need to notice that we can apply Theorem \ref{curves} for a multiple of $U_1^F$ in a neighborhood of the origin, and then by scaling (see \eqref{UFt}) it follows that as $t \to \infty$,
$$t^{-\frac \beta 2} U^F_t \to \mu \, \mathcal U_0 \quad \mbox{in} \quad C^1(B_1), \quad \quad t\frac {d}{dt} (t^{-\frac \beta 2} U^F_t) \to 0 \quad \mbox{uniformly in $B_1$},$$
for some positive constant $\mu>0$.

Finally, the proofs of Theorem \ref{localclassintro} and \ref{thm:best} now follow as in the case $F=(1,1,1).$

 \section{Higher dimensions}\label{sec:higher}
 In this section we assume we are in dimension $n\geq 3.$ We denote points by $$x=(x', x'') \quad\mbox{with} \quad x'=(x_1,x_2), \quad x''=(x_3, \ldots, x_n)$$ and by $(r, \theta)$ the polar coordinates in the $(x_1,x_2)$ plane. Also,
 $$\mathcal C_r:= B_r' \times B_r'', \quad \mathcal C_r \setminus \{0'\}= \mathcal C_r \setminus (\{0'\} \times \R^{n-2})$$

 We still denote by 
 $$\mathcal U_0= r^{3/2}\cos \frac 3 2 \theta$$
 and (from Subsection \ref{subsec:global})
 $$U_0^F=U_0^F(x')$$
 the trivial extension of the 2d homogeneous cones to $\R^n.$
 
 Let $U$ be a minimizer of $J(\cdot, F)$ with $F>0$ constant.
 \begin{defn} We say that $0 \in \Gamma(U)$ is a $Y$-singularity if it has a  2-dimensional blow-up. Precisely, set $a_r:=\max_{B_r} |U|,$ and either assume that for a sequence $r\to0$ 
$$\lim_{r\to0} \frac{a_r}{r^2} = +\infty$$
and (up to a rotation) $$U_r(x):= a_r^{-1} U(rx) \to \mathcal U_0 \quad \text{a type I Y-singularity}$$
or 
$$\liminf_{r\to0} \frac{a_r}{r^2} < +\infty$$
and (up to a rotation) $$U_r(x):= r^{-2} U(rx) \to U^F_0 \quad \text{a type II Y-singularity}.$$  \end{defn}
 
According to our definition, a free boundary point could simultaneously be a type I and type II Y-singularity.  
The goal of this section is to prove the uniqueness of blow-ups which also rules out this possibility. Precisely, our main results are as follows.

\begin{thm}\label{typeIgeneral} Let $x_0 \in \Gamma(U)$ be a type I Y-singularity. Then, the blow-up profile is unique and in a neighborhood of $x_0$ the free boundary $\Gamma(U)$ consists of three $C^{1,\alpha}$ hypersurfaces that meet at $120^\circ$ on a $C^{1,\alpha}$ codimension 2 curve made of type I Y-singularities (for any $\alpha<1/2$).
\end{thm}

\begin{thm}\label{typeIIgeneral} Let $x_0 \in \Gamma(U)$ be a type II Y-singularity. Then, the blow-up profile is unique. Moreover, all the triple junction points near $x_0$ belong to a $C^{1,\alpha}$ codimension 2 curve (for any $\alpha<1/2$). 
\end{thm}

The strategy of the proof is the same as in the 2D case. The main difference is that we need to obtain estimates for the linearized equations in the $n$-dimensional setting, which takes the following form:
\begin{align}\label{FBlinearnew}
& \Delta V=0 & \text{in $\{U_0^F\neq 0\}$}\\
\label{transmissionnew} & V_{\nu^+} + V_{\nu^-} =\frac{a^F}{|\nabla U_0^F|} V & \text{on $\{U_0^F=0\}.$}
\end{align}

First we establish the H\"older continuity of solutions on the ``singular axis" $x'=0$.

\begin{lem}\label{extendnew} Let $V \in \mathcal D$ be a bounded solution to \eqref{FBlinearnew}-\eqref{transmissionnew}in $ \mathcal C_1\setminus \{0'\}.$ 
Then, for some $\alpha>0$ small,  
\begin{equation}\label{V0new}|V(x)| \leq C |x'|^\alpha \|V\|_\infty \quad \mbox{in} \quad \mathcal C_{1/2} \setminus \{0'\},\end{equation}
hence $V$ can be extended continuously on the line $\{0'\} \times B_1''$ with zero value. Moreover, 
\begin{equation}\label{complete}
    \|V\|_{C^{\alpha}(\mathcal C_{1/2})} \leq C \|V\|_{L^\infty(\mathcal C_1)}.
\end{equation}
\end{lem}
The proof remains the same as Lemma \ref{extend}. We only need to replace annuli $B_{2r} \setminus B_{r/4}$ with cylindrical annuli $(B'_{2r} \setminus B'_{r/4}) \times B''_r(x_0'')$. 
The extension from \eqref{V0new} to \eqref{complete} follows from interior  estimates for transmission problems (see Remark \ref{linearsmooth}).
Notice that the interior estimates can be applied at all scales since \eqref{V0new}-\eqref{complete} is invariant under dilation due to the fact that $U_0^F$ is homogeneous of degree 2. 

Next we obtain the asymptotic expansion of a solution to the linearized problem near the origin.

\begin{lem}[Expansion in terms of homogeneous solutions]\label{expandV2} 
Let $V \in \mathcal D$ be a bounded solution to \eqref{FBlinearnew}-\eqref{transmissionnew}in $ \mathcal C_1\setminus \{0'\}.$ 
Then, for $\alpha$ given in Lemma \ref{extendnew},

\begin{equation}\label{Vx}V(x)=({\bf m}+ M x'') \cdot \nabla U^F_0 + p |x'|^{\lambda_3}g_3^F(\theta) + q \p_\theta U^F_0 + O(|x|^{2+\alpha}),\end{equation}
where $M$ is a linear map from the $x''$ to the $x'$ variables, and $\|M\|$, $|{\bf m}|, |p|$ and $|q|$ are universally bounded.
\end{lem}
\begin{proof} Equations \eqref{V0new}-\eqref{complete} is invariant under translations in the $x''$ direction. Thus, by using discrete differences and iterating \eqref{complete} we obtain that all the $x''$-derivatives satisfy \eqref{V0new}-\eqref{complete}, that is 
\begin{equation}\label{V0newder}|D^k_{x''}V(x)| \leq C |x'|^\alpha \|V\|_\infty \quad \mbox{in} \quad \mathcal C_{1/2},\end{equation}
and
\begin{equation}\label{completeder}
    \|D^k_{x''}V\|_{C^{\alpha}(\mathcal C_{1/2})} \leq C(k) \|V\|_{L^\infty(\mathcal C_1)}.
\end{equation}
This means that $V(x',0)$ solves the linearized problem in 2D 
\begin{align}\label{9}
& \Delta_{x'} V= \Phi & \text{in $\{U_0^F\neq 0\}$}\\
\label{10} & V_{\nu^+} + V_{\nu^-} =\frac{a^F}{|\nabla U_0^F|} V & \text{on $\{U_0^F=0\}.$}
\end{align}
and $$\Phi=-\Delta_{x''} V(x',0) \in \mathcal D,$$ and by \eqref{V0newder}-\eqref{completeder} \begin{equation}\label{00}|\Phi| \leq C|x'|^\alpha, \quad \|\Phi\|_{C^{\alpha}(\mathcal C_{1/2})} \leq C.\end{equation}

\medskip
\noindent
{\it Claim:} Solutions to \eqref{9}-\eqref{10} satisfy the same expansion as solutions to the 2D homogeneous problem (that is $\Phi=0$) with an error of order $O(|x'|^{2+\alpha}).$ 

Indeed, when we compare in $B_r'$ our solution with the solution to the homogeneous problem and the same boundary data on $\p B'_r$, their difference is bounded by $Cr^{2}\|\Phi\|_\infty \leq Cr^{2+\alpha}$ by the maximum principle. This is because the absolute value of the difference, which is single-valued, has Laplacian bounded below  by $-\|\Phi\|_\infty$, as the distributional part along the rays is nonnegative due to the fact that $a_F\ge 0$ in the transmission condition. This means that the Campanato method of successive iterations gives the desired claim. Therefore, (see Lemma \ref{L8.2}),
$$V(x',0)=V_{{\bf m}, p, q}^F + O(|x'|^{2+\alpha}).$$ In the same way, by the translation invariance, also the derivatives in the $x''$ direction have a similar expansion, hence 
$$D_{x''}V(x',0)= \nabla U_0^F \cdot M + O(|x'|^{1+\alpha}),$$ for some linear map $M$ from the $x''$ to the $x'$ variables. Since 
$$D_{x''}^2V=O(|x'|^{\alpha}),$$ we can integrate in the $x''$-direction and obtain the desired expansion.
\end{proof}

\begin{rem} The expansion of Lemma \ref{expandV2} shows that the value of $\alpha$ in Lemma \ref{extendnew} can be improved to $\alpha=1$. This means that the Campanato iteration gives the expansion \eqref{Vx} with an error $O(|x|^{2+\alpha})$ for any $\alpha <1/2.$
\end{rem}

We now observe that the improvement of flatness Proposition \ref{IFgeneral} holds in the $n$-dimensional setting of this section. Here $U_t^F$ is the trivial extension in $\R^n$ of the 2D family of global solutions in Section \ref{sec:non-sym-F}.

\begin{prop}[Improvement of Flatness in dimension $n\geq 3$]\label{IFgeneraln}Let $U$ minimize $J(\cdot, F)$ in $B_1$, $F > 0$ constant. There exists $\eps_0, \delta, \rho, C$ depending on $F$, such that if for $\eps \leq \eps_0$
$$d(U,U^F_t) \leq \eps, \quad |t| \leq \delta, \quad \text{in $B_1$}$$
then $$d(U,\tilde U^F_{t+\eps p}) \leq \frac \eps 2 \rho^2 \quad \text{in $B_\rho(\eps \bar{\bf m})$},$$ where $\tilde U^F_{t+\eps p}$ denotes a rigid motion in $\R^n$ of $U^F_{t+\eps p},$ by a rotation $O_{\eps A}$ and translation by $\eps \bar{\bf m}$, and $|A|$, $|\bar{\bf m}|$, $|p| \leq C$, $C$ universal.
\end{prop}
The proof is the same as before. Indeed, let $A$ be an antisymmetric 
$n\times n$ matrix and let $O_{\eps A}$ be the orthogonal transformation near $I$ in the direction of $A$, $$O_{\eps A}=I+\eps A+O(\eps^2).$$ Then,
$$\tilde U^F_{t+\eps p}(x)=U^F_{t+\eps p}(\eps {\bf m} + O_{\eps A}x)= U_t^F(x)+ \eps[\nabla U_t^F(x)\cdot ({\bf m}+Ax)]+ \eps p (\p_t U_t^F(x)) + o(\eps).$$

\
Using the fact that $U_t^F$ is $C^1$ in the $(x,t)$-variables, we conclude that 

$$\frac{\tilde{U}^F_{t+\eps p} - U^F_t}{\eps} 
= ({\bf m}+ M x'') \cdot \nabla U^F_0 + p |x'|^{\lambda_3}g_3^F(\theta) + q \p_\theta U^F_0 + o(1)$$
with
$$o(1) \to 0, \quad \text{as $\eps,\delta \to 0.$}$$ Here $q$ is the $-a_{12}$ entry in $A$, while $M$ is the restriction of $A$ to the first two rows and the last $n-2$ columns.

Similarly, we obtain the corresponding results for type I cones in the $n$-dimensional setting, which we state for completeness. First, the linearized problem takes the form 
\begin{align}\label{FBBB}
& \Delta V=0 & \text{in $\{\mathcal U_0\neq 0\}$}\\
\label{transmissionBBB} & V_{\nu^+} + V_{\nu^-} = 0 & \text{on $\{\mathcal U_0=0\}.$}
\end{align}

Now the third eigenfunction is given by $\mathcal U_0$ and the asymptotic expansion takes the following form.
\begin{lem}[Expansion for \eqref{FBBB}-\eqref{transmissionBBB}]\label{expandV20} 
Let $V \in \mathcal D$ be a bounded solution to \eqref{FBBB}-\eqref{transmissionBBB} in $ \mathcal C_1\setminus \{0'\}.$ 
Then, 
\begin{equation}\label{Vx2}V(x)=({\bf m}+ M x'') \cdot \nabla U_0 + p \mathcal U_0 + q \p_\theta \mathcal U_0 + O(|x|^{2+\alpha}),\end{equation}
where $M$ is a linear map from $x''$ to $x'$, $\|M\|$, $|{\bf m}|, |p|$ and $|q|$ are universally bounded.
\end{lem}
The improvement of flatness can be stated as follows.

\begin{prop}[Improvement of Flatness for type I cones]\label{zeron} Let $U$ minimize $J(\cdot, F)$ in $B_1$ with $F \geq 0$. There exist $\eps_0, \rho>0$ universal such that if for some $\eps \leq \eps_0$
\begin{equation}\label{ass01}d(U,\mathcal U_0) \leq \eps, \quad \|F\|_{\infty} \leq \eps^2 \quad \text{in $B_1$,}\end{equation} 
then $$d(U,(1+\eps p)\tilde {\mathcal U}_{0}) \leq \frac \eps 2 \rho^{3/2} \quad \text{in $B_\rho(\eps \bar{\bf m})$},$$ with $\tilde {\mathcal U}_0$ a rigid motion in $\R^n$ (a rotation $O_{\eps A}$ and translation by $\eps \bar{\bf m}$) of $\mathcal {U}_{0},$ and $|A|,|\bar{\bf m}|,|p| \leq C$, $C$ universal.
\end{prop}

Now it is standard to obtain Theorems \ref{typeIgeneral} and \ref{typeIIgeneral} by iterating Propositions \ref{zeron} and \ref{IFgeneraln} as in the previous sections.

%%%%%%%%%%
%%%%%%%%%%

\section{The proof of Lemma \ref{impro}.} \label{sec:Proof-Lemma-impro}

We restate the lemma below.

\begin{lem}\label{improAPP} Let $u_0 \in C^1(B_1)$ with $|\nabla u_0| \neq 0$, satisfy
$$\Delta u_0 = a^+\chi_{\{u_0>0\}} -a^- \chi_{\{u_0<0\}} \quad \text{in $B_1$},$$ with $a^\pm \in \R.$ If $u$ satisfies the same equation,
$$\Delta u = a^+\chi_{\{u>0\}} - a^-\chi_{\{u<0\}} \quad \text{in $B_1$},$$
and
$$|u-u_0| \leq \eps \quad \text{in $B_1$}$$ then set 
$$v_\eps:= \frac{u-u_0}{\eps},$$ as $\eps \to 0$, $v_\eps \to v$ uniformly on compact subsets of $B_1$ (up to a subsequence), and $v$ satisfies the transmission problem 
$$\Delta v=0 \quad \text{on $\{u_0\neq 0\}$}$$
$$v_{\nu^+} + v_{\nu^-} =\frac{a^++a^-}{|\nabla u_0|} v \quad \text{on $\{u_0=0\}$}.$$ Here $\nu^\pm$ are the unit normals to $\{u_0=0\}$ pointing to $\{u_0>0\}$ and $\{u_0<0\}$ respectively.
\end{lem}
\begin{proof} We notice that since $\Delta u_0$ is bounded, the function $u_0 \in C^{1,\alpha}$, for any $\alpha < 1$. Since $|\nabla u_0| \neq 0$ the level set $\Gamma_0:=\{u_0=0\}$ is a $C^{1,\alpha}$ hypersurface. Given that $|u-u_0|\leq \eps$, the level set $\Gamma:=\{u=0\}$ is in an $O(\eps)$- neighborhood of $\Gamma_0.$ Away from this neighborhood, clearly $v_\eps$ is harmonic and bounded, therefore on any compact set away from $\Gamma_0$, up to extracting a subsequence, $v_\eps \to v$ uniformly and $v$ is harmonic. It remains to prove that the convergence holds in a neighborhood of $\Gamma_0$ and that the limiting function satisfies the transmission condition on $\Gamma_0.$

For simplicity, say $0\in \Gamma_0$ and assume that $\nabla u_0(0)$ points in the $e_n$ direction. We pick a small enough ball $B_\rho$ so that $\p_n u$ and $\p_n u_0$ are bounded below in this ball by a fixed positive  constant. Notice that since also $u \in C^{1,\alpha}$ and $|u-u_0| \leq \eps$, it follows that $|\nabla(u-u_0)|$ is sufficiently small.

Now, in $B_\rho$ we perform a hodograph transform by viewing the graphs of $u$ and $u_0$ as graphs of the function $w$ and $w_0$, respectively,  in the $x_n$-direction.  
That is, 
$$x_{n+1}= u(x_1, \ldots, x_n) \quad \Longleftrightarrow \quad x_n=w(x_1,\ldots, x_{n-1}, u)$$
and similarly for $u_0, w_0.$ 
Then by differentiating the right hand side of this equivalence we obtain the second derivatives of $u$ in terms of the second derivatives of $w$ and conclude that 
$$\Delta u= a^{ij}(\nabla w) w_{ij}$$
and similarly for $u_0, w_0.$ Here $a^{ij}(\nabla w)$ is uniformly elliptic and since $w$ is $C^{1,\alpha}$ these coefficients are H\"older continuous. Thus, our equation takes the form
$$\mathcal L w:=a^{ij}(\nabla w) w_{ij} = a^+\chi_{\{x_n>0\}} - a^-\chi_{\{x_n<0\}},\quad \text{in $B_{c\rho}$}$$ (similarly for $w_0)$ in which the right hand side has a discontinuity along a hyperplane. From here we can conclude that $w$ is $C^{1,1}$ in $B_{c\rho/2}$ and $D^2 w$ is Hölder continuous on either side $\{x_n \geq 0\}$ and $\{x_n \leq 0\}$ up to $\{x_n=0\}$ with a discontinuity only in the $w_{nn}$ term which accounts for the discontinuity of the right hand side (see Remark \ref{rem85} below).

Now we look at 
$$\phi_\eps:=\frac{w-w_0}{\eps}$$
which solves the equation 
\begin{equation}\label{phieps}
    \tilde{a}^{ij}(x)(\phi_\eps)_{ij} + \tilde{b}(x)\cdot \nabla\phi_\eps=0\end{equation}
with $$\tilde{a}^{ij}(x):=a^{ij}(\nabla w_0) \in C^{0,\alpha}$$
and  
$$ \tilde b(x):=\int_0^1 \nabla_p a^{ij}(t\nabla w + (1-t)\nabla w_0)w_{ij} dt \in L^\infty.$$
In particular, this equation has $C^{1,\alpha}$ estimates (for all $\alpha<1$) and we get that $$\|\phi_\eps\|_{C^{1,\alpha}(B_{c\rho/2})} \leq C \|\phi_\eps\|_\infty \leq C.$$

In conclusion, the function $u$ is obtained from $u_0$ by the domain deformation 
\begin{equation}\label{dd}
    u_0(x)=u(x+\eps \phi_\eps e_n),\end{equation}
with $\phi_\eps$ a function with bounded $C^{1,\alpha}$-norm. Here the pullback is still denoted by $\phi_\varepsilon$.
Therefore, \begin{equation}\label{veps}v_\eps(x) = -(\p_n u(x)) \phi_\eps(x) + O(\eps),\end{equation}
which, as $\eps \to 0,$ proves the convergence of $v_\eps$ in a neighborhood of $\Gamma_0$ to a function $$v=-(\p_n u_0) \phi$$ which is harmonic away from $\Gamma_0.$ Now it is easy to check that the transmission condition is satisfied. To this end, since $\phi $ is $C^1$, 
it suffices to show that $\p_n u_0$ satisfies the condition. In fact, we show that every directional derivative of $u_0$ satisfies the condition, and it suffices to show it at the origin where $\nu^+=e_n$ and $\nu^-=-e_n$. 
Let ${\bf e}$ be a unit direction. Given that the only second derivative with a discontinuity at the origin is $(u_0)_{nn}$ we have that
$$(u_0)_{{\bf e}\nu^+} + (u_0)_{{\bf e}\nu^-}= ({\bf e} \cdot e_n)((u_0)_{\nu^+\nu^+} - (u_0)_{\nu^-\nu^-})$$
$$=({\bf e}\cdot e_n)(a^++a^-)= \frac{(u_0)_{{\bf e}}}{|\nabla u_0|}(a^++a^-).$$
\end{proof}

\begin{rem}
The proof above shows that the  function $v$ is $C^{1,\alpha}$ on either side of the surface $\Gamma_0$ and the transmission condition is satisfied in the classical sense.
\end{rem}

\begin{rem} Equation \eqref{veps} shows that the more general statement of Lemma \ref{improgeneral} holds. \end{rem}

\begin{rem}\label{remlast} The conclusion \eqref{dd}-\eqref{veps} that $u$ is obtained as a domain deformation of $u_0$ continues to hold if $u_0$ satisfies the equation with an $O(\eps)$ error. The only difference is that the equation \eqref{phieps} has a bounded right hand side and the $C^{1,\alpha}$ estimates still apply. Also the full conclusion of the Lemmas \ref{impro}-\ref{improgeneral} holds if the error is $o(\eps)$.
\end{rem}

\begin{rem}\label{rem85} Let $u$ be a solution to a linear equation with H\"older coefficients and a piecewise H\"older continuous right hand side with discontinuity along a hyperplane:
$$a^{ij}(x)u_{ij}=f(x) \quad \mbox{in $B_1$,} \quad a^{ij} \in C^\alpha(B_1), \, \, f \in C^\alpha(\overline B_1^-) \cap C^{\alpha}(B_1^+).$$
Then $u \in C_{loc}^{1,1}(B_1)$. Indeed,
$$v:=u - m_0 (x_n^+)^2,$$
with $m_0$ a constant depending on the jump of $f$ at $0$ and $a^{nn}(0)$, solves the equation with a right hand side that is pointwise $C^\alpha$ at $0$. Thus $v$ is pointwise $C^{2,\alpha}$ at the origin by the results in the book of Caffarelli-Cabre \cite{CC}.

\end{rem}

\section{Appendix}\label{App}
In this appendix, we prove the results from Section \ref{sec:Basic}. First, we recall the statement of the Alt-Caffarelli-Friedman formula as it appears in \cite{CJK}.

\begin{prop}\label{ACF} Suppose that $u_\pm$ are non-negative continuous functions in $B_1$. Suppose $\Delta u_\pm \geq -1$  in the sense of distributions and $u_+(x) \cdot u_-(x) =0$ for all $x \in B_1.$ Then, there is a dimensional constant $C>0$ such that 
\begin{equation}\label{ACF1}\Phi(r) \leq C\left(1+  \int_{B_1}\frac{|\nabla u_+|^2}{|x|^{n-2}} dx + \int_{B_1}\frac{|\nabla u_-|^2}{|x|^{n-2}}dx\right)^2, \quad 0<r \leq 1,\end{equation} where
$$\Phi(r)= \left(\frac{1}{r^2}\int_{B_r}\frac{|\nabla u_+|^2}{|x|^{n-2}}dx\right)\left(\frac{1}{r^2}\int_{B_r}\frac{|\nabla u_-|^2}{|x|^{n-2}}dx\right). $$
\end{prop}

\subsection{Optimal Regularity and Perturbative Estimates}\label{App:optinal} We now prove the optimal regularity of minimizers, Theorem \ref{OR}. 

\begin{proof}[Proof of Theorem \ref{OR}] We divide the proof in two steps. Notice that since $F \in L^\infty$ then, after dividing by a large constant we may assume that
$$\|F\|_{L^\infty} \le 1, \quad \|U\|_{L^2} \le 1.$$
 From the Euler-Lagrange equation we get
\begin{equation}\label{BoundedLaplace}
\Delta u^i \geq -1, \quad \Delta \hat u^i \leq 1, \quad \text{in $B_1$.}
\end{equation} 

{\it Step 1.} $u^i$ is continuous and $\{u^i>0\}$ is open. In particular, from \eqref{BoundedLaplace}
\begin{equation}\label{Bpositive}|\Delta u^i| \leq 1 \quad \text{in $B_1 \cap \{u^i>0\}$}\end{equation}
which implies that $u^i \in C^{1,\alpha}$ for all $0<\alpha<1$ in its positivity set.

\smallskip

From the first equation in \eqref{BoundedLaplace} we deduce that $u^i$ plus a quadratic polynomial is subharmonic hence 
\begin{equation}\label{BoundedLaplace2}
0 \le u^i \leq C \quad \text{in $B_{3/4}$,} 
\end{equation} for a dimensional constant $C>0.$ Moreover, $u^i$ can be defined pointwise as
$$u^i(x)=\lim_{r\to 0} \fint_{B_r(x)} u^i.$$
and it is upper semicontinuous. Notice that the limit is finite since $u^i \ge 0$. Consequently, $\hat u^i$ is defined pointwise and it satisfies
$$\hat u^i(x)=\lim_{r\to 0} \fint_{B_r(x)} \hat u^i.$$ Since, $\hat u^i$ satisfies the second equation in \eqref{BoundedLaplace} we deduce from this last identity that $\hat u^i$ is lower semicontinuous. On the other hand, $u^i=(\hat u^i)^+$ which implies that it is lower semicontinuous as well. We finally conclude that $u^i$ is continuous.

{\it Step 2.} $u^i$ is Lipschitz continuous.

We assume without loss of generality that $n\geq 3$, otherwise we extend the functions trivially in one extra dimension. If $u^i(0)=0,$ then, we deduce from \eqref{BoundedLaplace} that for all $0 < r \leq 1/2,$
\begin{equation}\label{Bounds}
\frac{c}{r^2}  \fint_{B_{r}}(u^i)^2 dx + O(1) \leq \frac{1}{r^2} \int_{B_r}\frac{|\nabla u^i|^2}{|x|^{n-2}} dx \leq \frac{C}{r^2} \fint_{B_{2r}}(u^i)^2 dx + O(1).
\end{equation}
Indeed, in view of \eqref{Bpositive}-\eqref{BoundedLaplace2}, on the set $\{u^i>0\}$
\begin{equation}\label{equivalent} \frac 1 2 \Delta ((u^i)^2) = |\nabla u^i|^2+ u^i \Delta u^i =  |\nabla u^i|^2+O(1),\end{equation}
in the sense of distributions. On the other hand, $\Delta ((u^i)^2)$ as a distribution in $B_1$ is supported on $\{u^i>0\}.$ This can be seen by approximation. We write the identity above with $u^i$ replaced by the truncation $(u^i-\eps)^+$, which is Lipschitz in view of \eqref{Bpositive}, then pass to the limit for $\eps \to 0.$

Now, denote by  $\Gamma_\rho \in C^{1,1}$ the smoothing of the fundamental solution $\Gamma(x):=|x|^{2-n}$ which coincides with it outside $B_\rho$, and is the appropriate quadratic polynomial on $B_\rho$. Then,
$$\Delta \Gamma_\rho = c(n) \frac{\chi_{B_\rho}}{|B_\rho|},$$
$$\Gamma_\rho - \Gamma_r \leq \Gamma, \quad \rho < r,$$ and $\Gamma_\rho - \Gamma_r$ is supported in $\overline{B_r}$.
Using the latter inequality and \eqref{equivalent} to estimate $\frac{1}{r^2} \int_{B_r}\frac{|\nabla u^i|^2}{|x|^{n-2}} dx$ then passing to the limit as $\rho \to 0$, we obtain the lower bound in \eqref{Bounds}.

For the upper bound, we cut off $|x|^{2-n}$ in the annulus $B_{2r} \setminus B_r$,
use \eqref{equivalent}, and integrate by parts.

Now, let $0 \in \p\{|U|>0\}$, so that \eqref{Bounds} holds for all $i$'s. Then applying Proposition \ref{ACF} to any pair $u^i, u^j$ we conclude that for all $k \in \{1,\ldots, N\}$, except at most one, say $k=1$,
$$\frac{1}{r^2} \int_{B_r}\frac{|\nabla u^k|^2}{|x|^{n-2}} dx \leq C, \quad k \geq 2, \quad 0<r \leq 1/4.$$ Notice that to bound the right hand side of \eqref{ACF1}, we used the bound in \eqref{Bounds} for $r=1/4.$

From the lower  bound in \eqref{Bounds} we deduce that (for a $C>0$ changing from line to line)
\begin{equation}\label{average}\frac{1}{r^2}\fint_{B_{r}}(u^k)^2 dx \leq C, \quad k\geq 2, \quad 0<r\leq 1/4.\end{equation}
Call, $$v^k(x):= \frac{u^k(rx)}{r}, \quad x\in B_{1}.$$
Then, from \eqref{BoundedLaplace} and \eqref{average} we have in $B_{1}$
$$\Delta v^k \geq - Cr, \quad \int_{B_{1}} (v^k)^2 dx \leq C,$$
therefore,
\begin{equation}\label{k2}v^k(x) \leq C \quad \text{in $B_{1/2}$, for all $k\geq 2.$}\end{equation}
We prove now that the inequality above also holds for $v^1.$ Indeed, from \eqref{BoundedLaplace} we have that
$$\Delta \hat{v}^1 \leq r \leq 1$$
and $$\hat{v}^1(0)=0.$$
 Then, by the mean value inequality for $\hat v^1$ we have
 $$C + \hat v^1(0) \geq \fint_{B_{1/2}} \hat{v}^1 dx = \fint_{B_{1/2}} (v^1 - \sum_{j \geq 2} v^j)dx$$
 and it follows from \eqref{k2} that
 $$\fint_{B_{1/2}} v^1 dx \leq C.$$ Since $v^1$ is positive and $v^1+C|x|^2$ is subharmonic, we deduce that $v^1 \leq C$ in $B_{1/4}.$

 In conclusion $$u^k \leq C r, \quad \text{in $B_{r/4}$, for all $k\geq 1$,}$$ which by interior regularity proves the Lipschitz continuity.
\end{proof}

\subsection{Perturbative estimates.}\label{App:perturb} Here we provide the proof of Proposition \ref{pert}. To this aim, we will use the following general lemma.

\begin{lem}\label{general} Let $0 \leq u \leq 1$ be a continuous function satisfying 
$$\Delta u \geq -1 \quad \text{in $B_1$}$$ and 
$$\{u>0\} \cap B_1 \subset\{|x_n| \leq \eps\}$$
for some $\eps>0$ small. Then,
$$ u \leq \eps^2 \quad \text{in $B_{1/2}$}.$$
\end{lem}
\begin{proof} By the maximum principle, 
$$u \leq w+\frac 12 (\eps^2 -x_n^2)^+\quad \text{in $B_1$}$$
where $w$ is the harmonic function in the strip $B_1\cap \{|x_n| <\eps\}$ which is zero on $\{|x_n|=\eps\} \cap B_1$ and 1 on the remaining part of the boundary, and is extended to zero outside the strip $\{|x_n| <\eps\}$. By Harnack inequality, 
$$\|w\|_{L^\infty(B_{1-k\eps})} \leq e^{-ck}.$$ In particular, $$w \leq e^{-c/\eps} \quad \text{in $B_{1/2}$},$$ and our claim follows.
\end{proof}

\noindent {\it Proof of Proposition \ref{pert}}. From assumption \eqref{flatnew}, after relabeling the indices, we can assume that
$$\{u^i>0\} \subset \{|x_n| \leq \eps\}, \quad \text{for all $i\geq 3$,}$$ and 
\begin{equation}\label{pos}u^1>0 \quad \text{on $\{x_n>\eps\}.$}\end{equation}
Thus, Lemma \ref{general} implies that
$$ u^i \leq \eps^2 \quad \text{in $B_{1/2}$, for all $i\geq 3$}.$$ In particular,
\begin{equation}\label{e2}|\hat u^1 - (-\hat u^2)| \leq 2 \sum_{i\geq 3} u^i \leq 2\eps^2 \quad \text{in $B_{1/2}$.}\end{equation}
Moreover, from the Euler Lagrange equation,
$$\Delta \hat u^1 \leq \|F\|_\infty, \quad \Delta (-\hat u^2) \geq - \|F\|_\infty, \quad \text{in $B_1$}.$$
Let $h$ be the harmonic replacement of $\hat u^1$ in $B_{1/2}$. By the maximum principle we get that in $B_{1/2}$
$$\hat u^1 \geq h - \|F\|_\infty,$$
and by \eqref{e2}, $(-\hat u^2) -2\eps^2 \leq \hat u^1=h$ on $\p B_{1/2}$ hence
$$(-\hat u^2) -2\eps^2 \leq h+\|F\|_\infty.$$

Therefore, again by \eqref{e2},
\begin{equation}\label{u1h}|\hat u^1 - h| \leq 4\eps^2 + \|F\|_\infty, \quad \text{in $B_{1/2}$}. \end{equation}
From \eqref{mu0}, \eqref{flatnew}, \eqref{pos}, and Hopf Lemma, we conclude that
$$u^1 \geq c(\mu)(x_n-\eps)^+.$$

Since $h$ is $2\eps$ close to $\hat u^1$ we can conclude that
$$|\nabla h(0)| \geq \frac 1 2 c(\mu).$$
Using that $h$ is well approximated by its linear part, and by choosing $\eps$ small enough, we get that after a homogeneous of degree one rescaling from $B_{r_0}$ to $B_1$ and multiplication by a constant, we may assume that
\begin{equation}\label{flat} |U-(x_1^+, x_1^-, 0,\ldots, 0)|\leq \eps_0, \quad \text{in $B_1,$}\end{equation}
and
\begin{equation}\label{smallF} \|F\|_\infty \leq \delta\eps_0,\end{equation}
for $r_0$ small depending on $\mu$ and $\delta, \eps_0$ universal.

\

Now the argument that led to \eqref{u1h} can be iterated indefinitely in a $C^{1,\alpha}$ fashion to obtain (for a similar argument see for example \cite{DFS}),
$$|U-a_k((x\cdot \nu_k)^+, (x \cdot \nu_k)^-, 0,\ldots, 0)|\leq \eps_0r^{1+\alpha}\quad \text{in $B_r$, $r=\rho^k$}$$ and $\rho$ universal, with
$$|a_{k+1}-a_k| \leq C\eps_0r^\alpha, \quad |\nu_{k+1}-\nu_k|\leq C\eps_0 r^{\alpha}.$$
\qed

\subsection{Sequences of minimizers.}\label{App:seq}
Here we exhibit the proof of Proposition \ref{compact}. First, the next two lemmas follow from standard techniques.
\begin{lem}[Lower semi-continuity]Let $U_k \in \mathcal A$ and $F_k \in L^2(B_1).$ Assume $U_k \to U, F_k \to F$ in $L^2$, then $$\liminf J(U_k,F_k) \geq J(U,F).$$
\end{lem}

\begin{lem}[Caccioppoli's Inequality]\label{lem:Caccioppoli} Let $U \in \mathcal A$ be a minimizer of $J(\cdot, F)$ in $B_1$, $F\in L^2(B_1)$. Then,
$$\int_{B_{1/2}} |\nabla U|^2 dx \leq C\left(\int_{B_1} |U|^2dx + \int_{B_1} |F|^2 dx\right),$$ for a dimensional constant $C>0.$
\end{lem}
\begin{proof} This follows by comparing the energy of $U$ with the energy of $(1-\phi) U$, where $\phi$ is a bump function which is zero outside $B_{3/4}$ and $1$ in $B_{1/2}.$
\end{proof}

With these tools in hand, we can prove the compactness result.

\begin{proof}[Proof of Proposition \ref{compact}] Let $V$ be an admissible competitor with $V=U$ in $B_1 \setminus \bar B_{1-\delta}$ for $\delta>0$ small. 
We use Lemma \ref{uv} below, and call $W_k$ the interpolation of $U_k$ and $V$ such that 
$$W_k= V \quad \text{in $B_{1-\delta}$}, \quad W_k=U_k \quad \text{in $B_1 \setminus \bar B_{1-\delta/2}$}.$$
Then, by the minimality of $U_k$ and Lemma \ref{uv}, we get
$$J_k(U_k, B_1) \leq J_k(W_k, B_1) \leq J_k(V, B_{r_k}) + J_k(U_k, B_1 \setminus \bar B_{r_k}) + o(1),$$
with $o(1) \to 0$ as $k\to \infty.$ Notice that, by Caccioppoli's inequality $\|U_k\|_{H^1}$, is uniformly bounded in $B_{1-\delta/2}$ and the hypotheses of Lemma \ref{uv} apply.

We deduce that,
$$J_k(U_k, B_{r_k}) \leq  J_k(V,  B_{r_k}) + o(1).$$
By the lower semi-continuity of $J$, since $F_k \to F$ in $L^2$, we obtain that,
$$J(U, B_{\bar r}) \leq  J(V, B_{\bar r}), \quad 1-\delta \leq \bar r \leq 1-\delta/2.$$ Our claim follows by letting $\delta \to 0.$
\end{proof}

In the next technical lemma, which we used in the previous proof, we interpolate between two functions which are $L^2$ close in an annulus, without increasing too much the total energy.
\begin{lem}\label{uv} Let $U_k,V_k$ be sequences in $\mathcal A$ and $\delta>0$ small. Assume that $U_k-V_k \to 0$ in $L^2(B_{1-\delta/2} \setminus \bar B_{1-\delta}),$ as $k\to \infty,$ and that $\|U_k\|_{H^1}, \|V_k\|_{H^1}$ and $\|F_k\|_{L^2}$ are uniformly bounded in $B_{1-\delta/2}$. Then, there exists $W_k \in \mathcal{A}$ with 
$$W_k:= \begin{cases}V_k \quad \text{in $B_{1-\delta}$}\\ U_k \quad \text{in $B_1 \setminus \bar B_{1-\delta/2}$}\end{cases}$$ such that, 

$$J_k(W_k, B_1) \leq J_k(V_k, B_{r_k}) + J_k(U_k, B_1 \setminus \bar B_{r_k}) + o(1),$$ 
for some $1-\delta \leq r_k \leq 1- \delta/2$,
 with $o(1) \to 0$ as $k \to \infty.$
\end{lem}
\begin{proof} 
First, we prove our claim in the case when in the annular region where we interpolate we have $V_k \leq U_k,$ that is the inequality holds for each component. For notational simplicity, in what follows we drop the subscript $k$. 

Let  $\varphi(s)$ be a smooth function on $\R$, which is $0$ on $(-\infty, 0]$ and $1$ on $[1,+\infty)$ and 
for $0< \mu< \frac{\delta}{4}$, set 
$$\varphi_\mu(s) = \varphi\left(\frac{s}{\mu}\right).$$
For all $r \in [1-\delta, 1-\frac 3 4 \delta]$, denote by $\varphi_r(x):=\varphi_\mu(|x|-r)$ and define 

\begin{equation}\label{w_r}
W_r:= \varphi_r U + (1-\varphi_r)V.
\end{equation} 
Since we are assuming that $V \leq U$, this guarantees that $W_r \in \mathcal A.$
Then,
\begin{align*}J(W_r, B_1) \leq &J(V, B_r) + J(U, B_1 \setminus \bar B_{r+\mu}) + \\
\ & C\int_{B_{r+\mu} \setminus \bar B_r} (|\nabla \varphi_r|^2|V-U|^2 + \varphi_r^2|\nabla U|^2+(1-\varphi_r)^2|\nabla V|^2)dx\\  + & \int_{B_{r+\mu} \setminus \bar B_r} (\varphi_r F\cdot U+ (1-\varphi_r)F\cdot V) dx.
\end{align*}
Thus (for $C$ changing from line to line),
\begin{align*}J(W_r, B_1) \leq &J(V, B_{r}) + J(U, B_1 \setminus \bar B_{r}) + CE(r)\end{align*}
with 
$$E(r):=\int_{B_{r+\mu} \setminus \bar B_r} (\frac{1}{\mu^2}|U-V|^2 + |\nabla U|^2+|\nabla V|^2 +|F||U|+|F||V|) dx$$

We now average in $E(r)$ on the interval $[1-\delta, 1-\frac 3 4 \delta]$ and use Fubini to get 
\begin{align*} \frac 4 \delta \int_{1-\delta}^{1-\frac 3  4 \delta} E(r) dr & \leq  \frac 4  \delta \mu \int_{B_{1-\delta/2} \setminus \bar B_{1-\delta}} (\frac{1}{\mu^2}|U-V|^2 + |\nabla U|^2+|\nabla V|^2 +|F||U|+|F||V|) dx.
\end{align*}
By choosing first $\mu$ small and then $k$ large, we obtain the desired statement in view of the assumption that $\|U_k-V_k\|_{L^2} \to 0$ as $k \to \infty$.

This proves the result when $V_k \leq U_k$. Otherwise, we do this procedure twice. First we interpolate $U_k$ and $\min\{U_k, V_k\}$ between $B_{1-\delta}$ and $B_{1-3\delta/4}$, and then we interpolate the $\min\{U_k,V_k\}$ and $V_k$ between $B_{1-7\delta/8}$ and $B_{1-\delta/2}.$
\end{proof}

We conclude the section with the proofs of Lemma \ref{H1loc} and Proposition \ref{prop:asympt}.

\begin{proof}[Proof of Lemma \ref{H1loc}] By compactness and optimal regularity, $U$ is Lipschitz continuous in $B_1$ and by the Euler-Lagrange equation $U \in C^{1,\alpha}(\{|U|>0\})$. The convergence is then obvious in the set $\{|U|>0\}.$

Let $x_0$ be a point where $|U(x_0)|=0$ and $U$ and all the $U_k$'s are differentiable. It suffices to show that $\nabla U_k(x_0) \to \nabla U (x_0)$, which implies the almost everywhere pointwise convergence of the gradients. 

Since $U\geq 0$, $\nabla U(x_0) = 0$. Hence, 
$$|U(x)| \leq \delta r, \quad \text{in $B_r(x_0)$, for all small $r$'s.}$$ Therefore, for a fixed small $r$ and for all large $k$'s,
$$|U_k(x)| \leq 2\delta r, \quad \text{in $B_r(x_0)$.}$$
By the optimal regularity estimate \eqref{opt} in $B_r(x_0)$,
$$\|\nabla U_k\|_{L^\infty(B_{r/2}(x_0))} \leq C(\frac 1 r \|U_k\|_{L^\infty(B_r(x_0))} + r\|F_k\|_{L^\infty(B_r(x_0))})$$
which combined with the inequality above gives, 
$$\|\nabla U_k\|_{L^\infty(B_{r/2}(x_0))} \leq C(3\delta + r\|F\|_{L^\infty(B_r(x_0))}) \leq \eps$$
by choosing first $\delta$ and then $r.$
\end{proof}

\begin{proof}[Proof of Proposition \ref{prop:asympt}] As in the proof of Lemma \ref{H1loc}, the rescaled Lipschitz estimate, gives that if
$$\liminf_{r\to 0} \frac{\|U\|_{L^\infty(B_r)}}{r}=0$$ then the first alternative holds. Otherwise, there exists $\mu>0$ such that for all small $r$'s, \begin{equation}\label{mu}\mu r \leq \|U\|_{L^\infty(B_r)} \leq \mu^{-1} r.\end{equation} Then, by compactness, the Lipschitz rescalings
$$U_r(x) := \frac{U(rx)}{r}$$ converge uniformly on compact sets to a global solution $U_0$ of $J_0,$  which satisfies \eqref{mu} as well for all $r$'s. By Almgren's frequency formula (Proposition \ref{almgren}) $U_0$ must be homogeneous of degree 1. Thus, the ACF monotonicity formula (Proposition \ref{ACF}) implies that $U_0$ must have either only one non-trivial component or two components given by $a_\pm(x\cdot \nu)^\pm.$ The former is not possible because that component, say $u^1_0$, would coincide with $\hat u_0^1$ and be superharmonic by the Euler Lagrange equation. However, it would be zero at the origin contradicting the maximum principle. Then, $U_0$ must have two components as above, and we end up in the assumptions of Proposition \ref{pert}.
\end{proof}

\noindent {\small \textbf{Acknowledgments.} 
 Ovidiu  Savin was supported by NSF, grant nr.  DMS 2349794; Henrik Shahgholian was supported by  Swedish Research Council, grant nr. 2025-03740.  
}

\end{document}